\documentclass[a4paper,11pt]{amsart}
\usepackage[utf8]{inputenc}
\usepackage[english]{babel}
\usepackage[bitstream-charter]{mathdesign}
\usepackage[T1]{fontenc}
\usepackage{amsfonts}
\usepackage{geometry}
\usepackage{amsmath, amsthm}
\usepackage{hyperref}
\usepackage{tikz}
\usetikzlibrary{arrows}
\usepackage{subfigure}
\usepackage{float}

\newcommand{\CC}{\mathbb{C}}
\newcommand{\PP}{\mathbb{P}}
\newcommand{\XX}{\mathbb{X}}

\newcommand{\caA}{\mathcal{A}}
\newcommand{\polyA}{\mathcal{P}}

\newcommand{\tic}{\mathcal{T}}

\newcommand{\Sing}{\mathrm{Sing}}

\newcommand{\caB}{\mathcal{B}}

\theoremstyle{theorem}
\newtheorem{theorem}{\sc Theorem}[section]  
\newtheorem{proposition}[theorem]{\sc Proposition}   
        
\newtheorem{lemma}[theorem]{\sc Lemma}                
\newtheorem{definition}[theorem]{\sc Definition}
\newtheorem{example}[theorem]{\sc Example}

\newtheorem{problem}{\sc Problem}

\theoremstyle{remark}

\newtheorem{remark}[theorem]{Remark}

\title{Unexpected curves arising from special line arrangements}


\author[M. Di Marca]{Michela Di Marca}
\address[M. Di Marca]{Dipartimento di Matematica, Universit\`a degli Studi di Genova, Genoa, Italy}
\email{dimarca@dima.unige.it}

\author[G. Malara]{Grzegorz Malara}
\address[G. Malara]{Department of Mathematics, Pedagogical University of Cracow, Krak\'ow, Poland}
\email{grzegorzmalara@gmail.com}

\author[A. Oneto]{Alessandro Oneto}
\address[A. Oneto]{Department of Mathematics, Universitat Polit\`ecnica de Catalunya, Barcelona, Spain}
\email{alessandro.oneto@upc.edu}

\begin{document}
	\maketitle

	\begin{abstract}
	In a recent paper \cite{CHMN17}, Cook II, Harbourne, Migliore and Nagel related {the splitting type} of a line arrangement in the projective plane to the number of conditions imposed by a general fat point of multiplicity $j$ to the linear system of curves of degree $j+1$ passing through the configuration of points dual to the given arrangement. If the number of conditions is less than the expected, {we say that the configuration of points admits} unexpected curves. In this paper, we characterize supersolvable line arrangements {whose dual configuration} admits unexpected curves
%	In particular, we prove that these are the ones where the number of lines is greater than twice the maximal multiplicity of \Alessandro{the singular points of the arrangement}. 
	{and we provide} other infinite families of line arrangements {with this property.}
	\end{abstract}
	
	\section{Introduction}
	Polynomial interpolation problems are among the most studied topics in algebraic geometry. A classical example deals with computing dimensions of linear systems of curves of given degree passing through a given set of points in the projective plane. In other words,  {on the space of coefficients of ternary homogeneous polynomials of degree $j$, we consider the system of linear equations given by imposing the vanishing at a set of $d$ points and we want to study the dimension of its solution. If the points are in {\it general position}, we may assume that this system of linear equations has maximal rank and the dimension of the solution is as small as possible, i.e., it is equal to ${j+2 \choose 2} - d$, unless this difference is negative, in which case it is zero \cite{GO81}.}
	
	If we consider points with some {{\it multiplicity}, usually called {\it fat points}}, where we require that the partial derivatives of the polynomial up to some order vanish at the points, the problem becomes much more complicated and only poorly understood. {In other words, we consider a polynomial interpolation problem where we look at plane curves having singularities of certain order at a set of points.} A complete answer is not known, even for points in {\it general position}. 

	Here is an example where the solution is not as expected. {Consider the space of plane quartics, which has dimension ${4+2 \choose 2} = 15$, and consider a scheme of five double points in general position, i.e., we consider the linear system of plane quartics having five singularities at general points.  Imposing a singularity at a point provides three linear equations, i.e., the vanishing of the three partial derivatives. Therefore, we have a system of $15$ linear equations on the space of plane quartics and we expect to have no quartics with five general singularities. However, through five general points there exists always a conic and, therefore, the double conic is an unexpected quartic singular at every point and, in particular, at the set of five general points.}
	
	The case up to nine general points goes back to Castelnuovo and it can also be found in the work of Nagata \cite{Nag60}. In the 1980s, Harbourne \cite{Har86}, Gimigliano \cite{Gim87} and Hirschowitz \cite{Hir89} independently gave conjectures on the dimension of a linear system of plane curves of given degree and with multiple general base points. {In \cite{CM01},} these conjectures have been proved to be all equivalent to an older conjecture by B. Segre \cite{Seg61} and, for this reason, we refer to them as the {\it SHGH Conjecture}. 
	
	In a recent paper \cite{CHMN17}, Cook II, Harbourne, Migliore and Nagel slightly changed the question. Instead of counting the number of linear conditions given by a set of general multiple points to the complete linear system of plane curves of given degree, as in the classical problem, they look at the conditions imposed by a general fat point to {the linear system of} plane curves of given degree and passing through some particular configuration of reduced points. 
%	In particular, they describe a configuration of nine points such that the general triple point fails to give the expected number of conditions on the space of quartics passing through such configuration. In this case, we say that the set of reduced points admits {\it unexpected} curves (see Example \ref{example:Brian}).
	
	This new question was motivated by previous works. Faenzi and Vall\`es noticed the relation between the {\it splitting type} of a line arrangement and curves passing through the set of points dual to the line arrangement and a fat point of multiplicity one less than the degree of the curve \cite{FV14}. Afterwards, Di Gennaro, Ilardi and Vall\`es gave an example of configuration of points admitting an unexpected curve \cite[Proposition 7.3]{DIV14}. {We recall it in Example \ref{example:Brian}. Actually, in \cite{DIV14}, the authors were studying {\it Lefschetz properties} of {\it power ideals}, i.e., ideals generated by powers of linear forms. In \cite{CHMN17}, the authors formalize the relation between Lefschetz properties of power ideals and the existence of unexpected curves for the configuration of points dual to the linear forms that define the power ideal.}
	
	In particular, in \cite{CHMN17}, the authors gave a characterization of the existence of unexpected curves for a given set of point in terms of the {\it splitting type} of the dual line arrangement. It is worth mentioning, that they also relate this problem to the famous {\it Terao's Conjecture} which claims that {\it freeness of a hyperplane arrangement depends only on the incidence lattice of the arrangement}. In particular, they show that, if the splitting types depend only on the combinatorics of the arrangement, or equivalently if the existence of unexpected curves depends only on the combinatorics of the configuration of the points, then Terao's Conjecture holds \cite[Corollary 7.11]{CHMN17}.
	
	\medskip
	In this paper, we characterize {\it supersolvable} line arrangements whose dual configuration of points admits unexpected curves. We also present several infinite families of line arrangements having this unexpected behavior by computing their splitting types. These families generalize examples from \cite{CHMN17, DIV14}.
	
	\subsection*{Formulation of the problem.}
	Let $S = \CC[x_0,x_1,x_2] = \bigoplus_{i \geq 0} S_i$ be the standard graded ring of polynomials with complex coefficients, i.e., $S_i$ is the $\CC$-vector space of homogeneous polynomials of degree $i$. {Any homogeneous ideal $I$ inherits the grading, i.e., $I = \bigoplus_{i\geq 0} I_i$, where $I_i = I \cap S_i$.
	
		The {\it fat point} of {\it multiplicity} $j$ and {\it support} at $P\in\PP^2$ is the $0$-dimensional scheme defined by the $j$-th power $\wp^j$ of the ideal $\wp$ defining the point $P$. We denote it by $jP$.}
	Observe that, a homogeneous polynomial $f \in S$ belongs to $\wp^j$ if and only if all partial derivatives of $f$ of order $j-1$ vanish at $P$. This gives ${j+1 \choose 2}$ linear equations, which justifies the following definition. 
	
	\begin{definition}\label{def:unexpected_curve}
		{Let $Z = P_1+\ldots+P_s$ be a set of reduced points in $\PP^2$. We say that $Z$ admits {\bf unexpected curves of degree $j+1$} if, for a general point $Q \in \PP^2$, we have that 
		$$
		\dim_{\CC}[I(Z+jQ)]_{j+1} > \max\left\{\dim_{\CC}[I(Z)]_{j+1} - {j+1 \choose 2}, 0\right\},
		$$
		where $I(Z+jQ) = I(Z) \cap I(Q)^j$. }
	\end{definition}
	The general problem in this theory is the following. 
	\begin{problem}\label{question:A}
		Classify all configurations of points $Z$ that admit unexpected curves.
	\end{problem}
	If $Z$ has general support, then it is well known that there are no unexpected curves of any degree. The following is an example coming from \cite{DIV14} of $9$ points which admits an unexpected quartic (see also \cite[Example 4.1.10]{Har17}).
		\begin{example}\label{example:Brian}
		The configuration is constructed, step by step, as follows (see Figure \ref{fig:1}). Consider four general points in the projective plane (the black dots). Then, there are three pairs of lines that contain all four points. 
%		(the solid, bold dotted and dashed lines). 
		Each pair has a singular point (the three dotted circles). Then, draw the line through two of these three points (the dotted line) and take the two points (the two white circles) where this line intersects the pair of lines whose singular point is the third point. This gives five additional points. The space of quartics passing through this configuration of points is $6$-dimensional, therefore, we expect to have no quartics with an additional general triple point (denoted as the two concentric blue circles). However, there exists an unexpected quartic. See Example \ref{example:Brian_coordinates} for an explicit construction in projective coordinates.
\end{example}
\begin{figure}\label{fig:1}
  \begin{center}
    \includegraphics[width=0.7\textwidth]{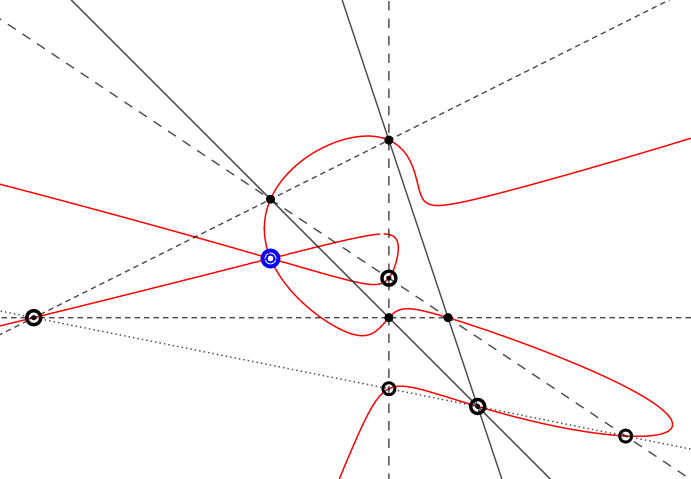}
  \end{center}
  \caption{The configuration unexpected quartic of Example \ref{example:Brian}.}
\end{figure}
{
In \cite{FGST18}, Farnik, Galuppi, Sodomaco and Trok show that this is, up to isomorphism, the only example of a configuration of points in the projective plane admitting an unexpected quartic.}

The configuration described in the example has a very special combinatorics {in relation to the $B_3$ arrangement (see \cite[Example 1.7]{OT}). We} describe it in more detail in the next section, see Example \ref{example:Brian_coordinates}. In \cite{CHMN17}, the authors connect the existence of unexpected curves for a configuration of points to the computation of the {\it splitting type} of the dual line arrangement $\caA_Z$ whose lines are defined by the linear equations having as coefficients the coordinates of the points in $Z$. We explain later in more detail these connection, but, in order to mention one of the main results in \cite{CHMN17} for the reader already familiar with these combinatorial concepts, a necessary condition for a set of points $Z$ to admit an unexpected curve in degree $j+1$ is that $a_Z \leq j \leq b_Z-2$, where $(a_Z,b_Z)$ is the splitting type of $\caA_Z$ \cite[Theorem 1.5]{CHMN17}.
	
	\smallskip
	In this paper, we generalize Example \ref{example:Brian} to infinite families of configurations having unexpected curves. In particular, while studying the problem, we noticed that the configuration given in the Example \ref{example:Brian} is the dual configuration of points to an arrangement of lines described in a paper of Gr\"unbaum \cite{Gru09} where the author explains particular families of (real) line arrangements. After some experiments with the algebra software {\it Macaulay2} \cite{Macaulay2} and {\it Singular} \cite{DGPS}, Gr\"unbaum's paper inspired us to find the examples we describe in this paper. The families of line arrangements that we consider here are {\it simplicial}, i.e., arrangements of lines where every cell is a triangle, or {\it near-simplicial}, i.e., sometimes we also have quadrilateral cells. As nicely explained in Gr\"unbaum's paper, these arrangements occur in the literature as examples and counterexamples in many contexts of algebraic combinatorics and its applications. In this case, we related them to a new interesting question on polynomial interpolation for plane curves.
	
	\subsection*{Structure of the paper.} In Section \ref{sec:basic}, we recall the basic notions and constructions of algebraic geometry and combinatorics that we need to analyse the problem. In Section \ref{sec:results}, we consider particular families of line arrangements that give unexpected curves. In Section \ref{appendix:further examples}, we provide more sporadic examples of line arrangements whose dual configurations have unexpected curves, but that we could not extend these to a general class of examples.
	
	\subsection*{Acknowledgements.}
	This project started during the ``2017 Pragmatic Summer School: Powers of ideals
and ideals of powers'' which was held at the University of Catania, Italy (June 19th - July 7th, 2017). We are grateful to the organizers (Alfio Ragusa, Elena Guardo, Francesco Russo and Giuseppe Zappal\'a) and the teachers (Brian Harbourne, Adam Van Tuyl, Enrico Carlini and T\`{a}i H\`{a}) of the school. In particular, we want to thank Brian Harbourne for suggesting and supervising this project and for useful comments on an early version of this paper. We also want to thank Michael Cuntz for sharing with us a database of crystallographic simplicial arrangements. The first author was partially supported by the "National Group for Algebraic and Geometric Structure, and their Applications" (GNSAGA-INdAM). The second author was partially supported by National Science Centre, Poland, grant 2016/21/N/ST1/01491. The third author was partially supported by  the Aromath team of INRIA Sophia Antipolis M\'editerran\'ee (France).
	\section{Basic notions and constructions}\label{sec:basic}
	
	In this section, we describe the main combinatorial objects we want to consider. For more details, we refer to the classical textbook on hyperplane arrangements by Orlik and Terao \cite{OT}.
	
	\subsection*{Dual line arrangement.}
	Given a configuration of reduced points $Z = P_1 + \ldots + P_d \subset \PP^2$, we consider the arrangement $\caA_Z$ of dual lines $L_1,\ldots,L_d$ in the dual space $(\PP^2)^\vee$. More precisely, if $P_i = (p_{i,0}:p_{i,1}:p_{i,2})$, for any $i = 1,\ldots,d$, then we define the line $L_i := \{p_{i,0}y_0 + p_{i,1}y_1 + p_{i,2}y_2 = 0\}$, where $T = \CC[y_0,y_1,y_2]$ is the coordinate ring of the dual plane. Moreover, if $\ell_i\in T_1$ is the linear form defining the line $L_i$, for any $i = 1,\ldots,d$, the arrangement $\caA_Z$ is defined by the polynomial $f_Z = \ell_1\cdots\ell_d \in T_d$.
	
	\begin{remark}
		When we say that a line arrangement admits unexpected curves we implicitly mean that the dual configuration of points admits unexpected curves, as defined in Definition \ref{def:unexpected_curve}.
	\end{remark}
	
	\subsection*{Splitting type of line arrangements.} Let $\caA$ be a line arrangement of $d$ lines and let $f_\caA \in T_d$ be the polynomial of degree $d$ defining it. We consider {the map defined by the gradient $\nabla_\caA = [\partial_{y_0}f_\caA,~\partial_{y_1}f_\caA,~\partial_{y_2}f_\caA]$}
	$$
		\mathcal{O}_{\PP^2}^3 \xrightarrow{\nabla_\caA} \mathcal{O}_{\PP^2}(d-1).
	$$ 
	We call the kernel of such a map the {\it derivation bundle} of $\caA$, i.e., the rank $2$ vector bundle $\mathcal{D}_\caA$ defined by
	$$
		0 \rightarrow \mathcal{D}_\caA \rightarrow \mathcal{O}_{\PP^2}^3 \xrightarrow{J_\caA} \mathcal{O}_{\PP^2}(d-1).
	$$
	{Up to a twist, the derivation bundle is isomorphic to the {\it syzygy bundle} of the {\it Jacobian ideal} of $f_\mathcal{A}$, i.e., the ideal $J_\mathcal{A} = (\partial_{y_0}f_\caA,\partial_{y_1}f_\caA,\partial_{y_2}f_\caA)$ generated by the first partial derivatives of the polynomial $f_\mathcal{A}$.}
\begin{definition}
	A line arrangement $\caA$ is said to be \textbf{free} with {\bf exponents}, or {\bf splitting type}, $(a_\caA,b_\caA)$ if $\mathcal{D}_\caA$ is free, i.e., if it splits as $\mathcal{D}_\caA = \mathcal{O}_{\PP^2}(-a_\caA) \oplus \mathcal{O}_{\PP^2}(-b_\caA)$.
\end{definition}
	{In general, the restriction of the derivation bundle on any line $\ell$ splits as $\mathcal{D}_\caA|_\ell = \mathcal{O}_{\PP^2}(-a) \oplus \mathcal{O}_{\PP^2}(-b)$. The splitting type $(a,b)$ is constant on a Zariski open subset of the dual projective plane, i.e., it is constant on a general line. We call this the {\it splitting type} of $\caA$ when the arrangement is not free. For details on these facts, we refer to \cite[Appendix]{CHMN17}.}
	
	If $\caA$ is a free line arrangement, then we have that the resolution of $S/J_{\caA}$ has length $2$. In particular, the resolution is
\[0 \rightarrow T(-(d-1)-a_\caA)\oplus T(-(d-1)-b_\caA) \rightarrow T(-(d-1))^3 \rightarrow T \rightarrow T/J_{\caA} \rightarrow 0,\]
where $a_\caA, b_\caA \in \mathbb{N}$ satisfy $a_\caA+b_\caA=d-1$. 

\begin{remark}
	The fact that the characteristic of the field does not divide $\deg(f_\caA)$ is crucial for this construction. The notion of a free line arrangement can be given more generally for any characteristic, but it is more complicated and it is not needed for the purposes of this paper. For this reason, in order to make the exposition clearer, we decided to give a definition which relies on the fact that we are in characteristic $0$ {and we refer to \cite{CHMN17} for the general case.}
\end{remark}

\begin{remark}
	If the line arrangement $\caA$ is actually the dual arrangement $\caA_Z$  of a configuration of points $Z$, we denote its splitting type by $(a_Z,b_Z)$.
\end{remark}

\subsection*{Conditions for unexpected curves.} 
Finally, we give the connection between {\it unexpected curves} for a configuration of points and the {\it splitting type of the dual line arrangement}. These are the main results in \cite{CHMN17} that motivated this project.

In \cite{FV14}, the authors associate to a set of reduced points $Z$ a {\it multiplicity index} defined as
$$
	m_Z := \min\{j~|~ \dim_{\CC}[I(Z+jQ)]_{j+1} > 0, \text{ for a general point } Q\}.
$$
In \cite[Lemma 3.5(i)]{CHMN17}, the authors associate directly the multiplicity index to the splitting type of the dual line arrangement $\caA_Z$. In particular, they proved that
\begin{equation}\label{eq: m = a}
	m_Z = \min\{a_Z,b_Z\}.
\end{equation}
Consequently, they obtain a characterization for configurations of points which admit unexpected curves. Here, another important numerical character is given by 
$
	t_Z := \min\left\{i ~|~ \dim_{\CC}[I(Z)]_{i+1} > {i+1 \choose 2}\right\}.
$
\begin{theorem}{\rm \cite[Theorem 1.1]{CHMN17}}
	Let $Z$ be a configuration of points in $\PP^2$ and let $\caA_Z$ be its dual line arrangement with splitting type $(a_Z,b_Z)$, say $a_Z \leq b_Z$. Then, $Z$ admits unexpected curves if and only if $a_Z < t_Z$. In this case, $Z$ admits an unexpected curve of degree $j+1$ if and only if $a_Z \leq j \leq b_Z - 2$.
\end{theorem}
Therefore, a solution to Problem \ref{question:A} is given by the following theorem.
\begin{theorem}{\rm \cite[Theorem 1.5]{CHMN17}}
	Let $Z$ be a configuration of points in $\PP^2$ and let $\caA_Z$ be its dual line arrangement with splitting type $(a_Z,b_Z)$. Then, $Z$ admits an unexpected curve of degree $j+1$ if and only if:
	\begin{enumerate}
		\item[\rm i.] $a_Z \leq j \leq b_Z - 2$;
		\item[\rm ii.] $\dim_{\CC}[I(Z)]_{t_Z} = {t_Z+1 \choose 2} - |Z|$,
	\end{enumerate}
	where $|Z|$ denotes the cardinality of the set $Z$.
\end{theorem}
From these results, it is clear that there is a close connection between the definition of unexpected curve for a set of points $Z$ and the splitting type of the dual line arrangement. Hence, our problem translates to a question about splitting types of line arrangements. 
{By \eqref{eq: m = a}, the splitting type can be computed with any algebra software by finding the least $j$ such that $[I(Z+jQ)]_{j+1} \neq 0$, for a general point $Q$. Unfortunately, this computation is very slow and inefficient because require to consider a field $\mathbb{F}$ which contains all the coordinates of the points in $Z$ and then, if $Q = (q_0:q_1:q_2)$, work over the field extension $\mathbb{F}(q_0,q_1,q_2)$.} 

In the next section, we focus on special line arrangements for which we can compute the splitting type and, consequently, deduce if the dual configuration of points admits an unexpected curve of some degree or not. Our computation mostly relies on the well-known Addition-Deletion Theorem {and, as far as we know, this is the only theoretical tool to compute the splitting type without doing it by direct computation.}

%\begin{definition}
%Given a line arrangement $\caA$ in the projective plane, we define:
%\begin{itemize}
%\item[\rm i.] $\Sing(\caA):=\{P \in \PP^2\ |\ P=L_i\cap L_j \mbox{ for some }L_i, L_j \in \caA\}$;
%\item[\rm ii.] for $P \in \Sing(\caA)$, $\ m(P):=\#\{L \in \caA \mbox{ such that } P\in L\}$;
%\item[\rm iii.] $m(\caA):=\max\{m(P)\ |\ P\in \Sing(\caA)\}$.
%\end{itemize}
%\end{definition}

We now recall a different version of the results in \cite{CHMN17} which is the precise way we use the aforementioned characterization of configurations of points having unexpected curves. 

	\begin{theorem}{\rm \cite[Theorem 1.2]{CHMN17}}\label{thm:unexpected_curve_condition_2}
		Let $Z\subset \PP^2$ be a finite set of points and let $(a_Z,b_Z)$ be the splitting type of the dual line arrangement, with $a_Z \leq b_Z$. Then, $Z$ admits an unexpected curve if and only if 
		\begin{enumerate}
			\item[\rm i.] $2a_Z + 2 < |Z|$;
			\item[\rm ii.] no subset of $a_Z + 2$ (or more) of the points is collinear.
		\end{enumerate}
		In this case, $Z$ has an unexpected curve of degree $j$ if and only if $a_Z < j \leq |Z| - a_Z - 2 = b_Z-1$.
	\end{theorem}
	
	{Theorem \ref{thm:unexpected_curve_condition_2} gives a criterion for {\it existence} of unexpected curves, but \cite{CHMN17} studies also conditions for {\it uniqueness} of such curves.
	
	\begin{proposition}\cite[Corollary 5.5]{CHMN17}\label{prop:unexpected_curve_unique}
		Let $Z \subset \PP^2$ be a finite set of points admitting unexpected curves and let $(a_Z,b_Z)$ be the splitting type of the dual line arrangement, with $a_Z \leq b_Z$. Then, $Z$ has a unique unexpected curve $C$ in degree $a_Z+1$. Moreover, for any $a_Z < j \leq b_Z-1$ the unexpected curves of degree $j$ are precisely the curves $C + L_1 +\ldots + L_r$, with $r = j - a_Z - 1$, where the $L_i$'s are arbitrary lines passing through the general point at which $C$ is singular.
	\end{proposition} 		
	In particular, when $Z$ admits unexpected curves, there is always a unique unexpected curve of degree $a_Z+1$.
	}
	
\section{Line arrangements with expected and unexpected behavior}\label{sec:results}
	While we were studying the problem and, in particular, Example \ref{example:Brian}, we noticed that {the dual line arrangement of the configuration of points appears under the name $A(9,1)$ in the list of simplicial line arrangements given by \cite{Gru09} (see also \cite{Cuntz} for updated list of these arrangements).} Here is the same example from this point of view.
	\begin{example}\label{example:Brian_coordinates}
	We construct a configuration of points in projective plane as described in Example \ref{example:Brian}. Consider the four vertices of a square: $(1:1:1),~(1:-1:1),~(-1:1:1)$ and $(-1:-1:1)$ and the intersection point of the diagonals of the square, i.e., the point $(0:0:1)$, and the intersections (at infinity) of the two pairs of parallel lines corresponding to the sides of the square, i.e., the points $(1:0:0)$ and $(0:1:0)$. Then, the line at infinity meets the two diagonals in two extra points $(1:1:0)$ and $(1:-1:0)$. Thus, we have a set $Z$ of nine points whose dual line arrangement $\caA_Z$ is defined by the polynomial $f = xyz(x+y+z)(x-y+z)(-x+y+z)(-x-y+z)(x+y)(x-y)$ and is depicted in Figure \ref{fig:A(9,1)}.
		
\begin{figure}[h]
	\begin{center}
			\subfigure[The configuration of points.]{
			\includegraphics{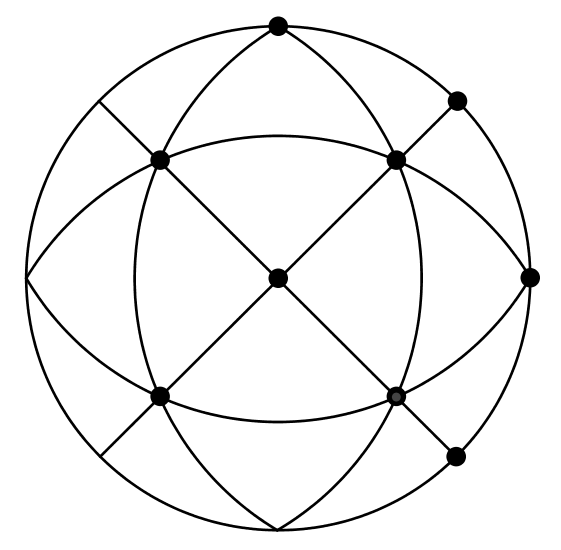}
			}~~~
			\subfigure[The dual line arrangement.]{
			\includegraphics{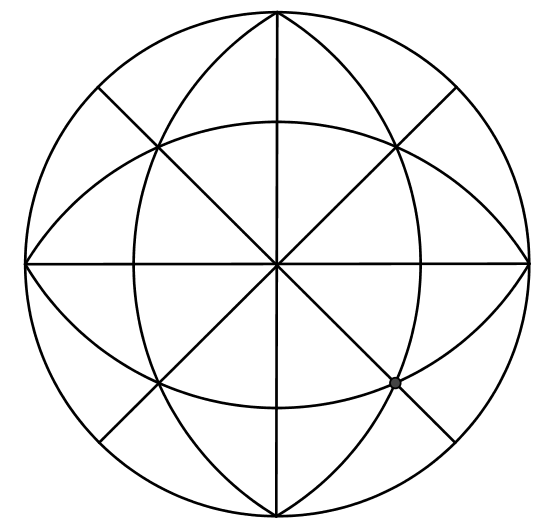}
			\label{fig:A(9,1)}
			}
	\end{center}
	\caption{The configuration of points in the projective plane and the dual line arrangement constructed in Example \ref{example:Brian_coordinates}. The pictures represent the projective plane and we use the classical model of the projective plane where the line at infinity is represented by a circle on which opposite points are identified. For this reason, some straight lines are represented by circular curves.}
\end{figure}
\end{example}
	
	Now, we look at families of line arrangements generalising the one constructed in Example \ref{example:Brian_coordinates}. In particular, we analyse their splitting type in order to establish for which arrangements the dual configuration of points admits unexpected curves of certain degrees.  
	
	\subsection{Supersolvable arrangements}
	We consider now a special family of line arrangements.
	\begin{definition}
		A line arrangement $\caA$ is called {\bf supersolvable} if there exists a {\bf modular} point, i.e., a point $P$ such that for every point $Q \in {\rm Sing}(\caA)$, the line joining $P$ and $Q$ is an element of $\caA$. 
	\end{definition}
	We denote the {\it multiplicity} of a point $P$ with respect to the arrangement $\caA$ as
	$
		m(P,\caA) := \left| \{\ell \in \caA ~|~ P \in \ell\} \right|.
	$ Moreover, we define $\Sing_k(\caA) := \{P\in \Sing(\caA) ~|~ m(P,\caA) = k\}$ and $\Sing_{\geq k}(\caA) := \bigcup_{i \geq k} \Sing_i(\caA)$.
	
	A useful property of supersolvable line arrangements is the following.
	\begin{lemma}{\rm \cite[Lemma 2.1]{AT16}}\label{lem:multiplicity_bound_supersolvable}
		Let $\caA$ be a supersolvable line arrangement. Let $P,Q\in{\rm Sing}(\caA)$ such that $P$ is modular and $Q$ is not. Then, $m(P,\caA) > m(Q,\caA)$. In particular, if a point has multiplicity 
		$$m(\caA)=\max \left\{m(P,\caA) \; | \; P \in{\rm Sing}(\caA) \right\},$$
		then it is modular.
	\end{lemma}
	
	\begin{definition}
		Let $\{\ell_1,\ldots, \ell_s \} \subset S_1$ be the set of the linear polynomials defining the lines of a supersolvable line arrangement $\caA$. We say that $\caA $ has {\bf full rank} if $\dim_{\CC}{\rm Span (\ell_1,\ldots, \ell_s )} =3.$
	\end{definition}
	
	In this section, we want to understand when supersolvable line arrangements admit unexpected curves. We give a necessary and sufficient condition to guarantee that a supersolvable line arrangement admits no unexpected curves and then we exhibit an infinite family of cases where we have unexpected curves. This family generalizes the configuration described in Example \ref{example:Brian_coordinates}. Our main tool is Theorem \ref{thm:unexpected_curve_condition_2} and, in order to use it, we need to compute the splitting type of supersolvable line arrangements. This is an easy application of the following well-known result which holds also in the more general setting of hyperplane arrangements. 
%
%\textcolor{red}{(Michela) We should be careful in this theorem, because the splitting type defined in [CHMN] assume $a_Z\leq b_Z$, but when we do addition deletion the order of $a$ and $b$ can change. Either we keep the word "exponents" to mean the non-ordered couple $\{a_Z, b_Z\}$ or we specify each time.}

	\begin{theorem}{\rm (Addition-Deletion Theorem; see \cite[Theorem 4.51]{OT})}\label{thm:addition-deletion}
		Let $\caA$ be a line arrangement in $\PP^2$ and $\ell \in \caA$. Let $\caA' := \caA \setminus\{\ell\}$. If the following conditions hold:
		\begin{enumerate}
			\item $\caA'$ is free and has splitting type $(a,b)$;
			\item $|{\rm Sing}(\caA) \cap \ell| = b+1$ (or $a+1$, respectively);
		\end{enumerate} 
		then, $\caA$ is free with splitting type $(a+1,b)$ (or $(a,b+1)$, respectively).
	\end{theorem} 
	
	Now, we can compute the splitting type for supersolvable line arrangements.
	
	\begin{lemma}\label{lem:splitting_supersolvable}
		Let $\caA$ be a supersolvable line arrangement where $d := |\caA|$ and $m := m(\caA)$. Then, the splitting type of $\caA$ is $(m-1,d-m)$.
	\end{lemma}
	\begin{proof}
		Let $O$ be a modular point with maximal multiplicity and consider $\caA = \caA_0 \cup \caA_1$, where $\caA_0$ is the subset of lines passing through the modular point $O$ and $\caA_1$ is the subset of lines not passing through $O$. Then, $m = |\caA_0|$ and denote $m' = |\caA_1|$. Namely, $d = m+m'$. We proceed by induction on $m'$.
		
		Let $m' = 0$. We have {that} $\caA = \caA_0$ is a central line arrangement given by $m$ lines passing through the point $O$. We compute the splitting type in this case by induction on $m$. If $m = 2$, it is easy to check that by definition the splitting type is equal to $(1,0)$. If $m>2$, by the Addition-Deletion Theorem and inductive hypothesis, we have that the splitting type of $\caA$ is $(m-1,0)$. 
		
		If $m' = 1$, let $\ell \in \caA_1$. Then, $$\Big|\bigcup_{\ell' \in \caA_0}\ell' \cap \ell\Big| = m$$ and, by the Addition-Deletion Theorem, we have that the splitting type of $\caA$ is $(m-1,1)$. If $m' > 1$, let $\caA_1 = \{\ell_1,\ldots,\ell_{m'}\}$. Then, we notice that, since $O$ is modular point, for every pair $\ell_i,\ell_j\in\caA_1$, the intersection $\ell_i\cap\ell_j$ lies on a line in $\caA_0$. Therefore, for each $i = 2,\ldots,m'$, if $\caA^{(i)} = \caA_0 \cup \{\ell_1,\ldots,\ell_{i-1}\}$, we have
		$$
			\Big|\bigcup_{\ell' \in \caA^{(i)}}\ell' \cap \ell_i\Big| = m,
		$$
	Therefore, by the Addition-Deletion Theorem and inductive hypothesis, we conclude that the splitting type of $\caA$ is $(m-1,m') = (m-1,d-m)$.
	\end{proof}
	
	We are now ready to give a necessary and sufficient condition for supersolvable arrangements to admit unexpected curves. Here, we denote $d := |\caA|$ and $m := m(\caA)$.
	\begin{theorem}\label{prop:supersolvable_unexpected_condition}
		A supersolvable line arrangement $\caA$ admits unexpected curves if and only if $d > 2m$, where $d$ is the number of lines and $m$ is the maximum multiplicity of a point of intersection of the lines of $\caA$. Moreover, if $d = 2m + 1$, there is a unique unexpected curve and it has degree $m$. 
	\end{theorem}
	\begin{proof}
		We use Theorem \ref{thm:unexpected_curve_condition_2}. We split the proof in {two cases: (1) $m-1\leq d-m$ and (2) $m-1>d-m$}. Observe that condition (ii) of Theorem \ref{thm:unexpected_curve_condition_2}, namely requiring to have no subset of $m_Z+2$, or more, collinear points in the configuration of points is equivalent to requiring that the multiplicity of the intersection points in the dual line arrangement is at most $m_Z+1$. Then:
		\begin{enumerate}
			\item if $m - 1 \leq d-m$, by Lemma \ref{lem:splitting_supersolvable}, $m_Z = m - 1$. Therefore, by Lemma \ref{lem:multiplicity_bound_supersolvable}, we may conclude that condition (ii) of Theorem \ref{thm:unexpected_curve_condition_2} is always satisfied.
Then, it is enough to observe that, since $m_Z = m-1$, condition (i) is equivalent to having $2m < d$;
			\item if $d-m < m-1$, from condition (ii) of Theorem \ref{thm:unexpected_curve_condition_2}, we get that the multiplicity of each intersection point in the line arrangement is at most $m_Z+1$. In particular, $m<m_Z+2=d-m+2$, hence $d>2m-2$. From the condition (i), we have instead that $2(d-m)+2<d$, so $d<2m-2$. As these two conditions are incompatible, this situation cannot occur.
		\end{enumerate}
		Note that, conversely, when $2m<d$ we are sure to be in the first case and then, the proof is concluded.
		
		{Uniqueness directly follows from Proposition \ref{prop:unexpected_curve_unique}.}
	\end{proof}
	
{
	\begin{remark}
		In a recent paper, Dimca and Sticlaru introduced the notion of {\it nearly supersolvable line arrangement} \cite{DS17}. They define a {\it nearly modular} point to be a point $P \in \Sing(\caA)$ such that:
		\begin{itemize}
			\item[(i)] for any point $Q \in \Sing(\caA)$, with the exception of a unique point of multiplicity $2$, say $P'$, $\overline{PQ}\in\caA$;
			\item[(ii)] $\overline{PP'} \cap \Sing(\caA) = \{P,P'\}$.
		\end{itemize} 
		Then, a line arrangement is {\it nearly supersolvable} if it has a nearly modular point. 
		
		Let $\caA$ be a nearly supersolvable line arrangement with $m = m(\caA)$ and $d = |\caA|$. In \cite[Corollary 3.2]{DS17}, Dimca and Sticlaru prove that the splitting type of $\caA$ is $(d-m,m-1)$, if $2m \geq d$, and $(\left\lfloor d/2 \right\rfloor,\left\lfloor d/2 \right\rfloor)$, if $2m < d$. By using the same idea of the proof of Theorem \ref{prop:supersolvable_unexpected_condition}, it follows that nearly supersolvable arrangements do not admit unexpected curves. 
	\end{remark}
}
	
	In the case of supersolvable real arrangements, we have the following propriety.
	
	\begin{proposition}{\rm \cite[Corollary 2.3]{AT16}}\label{prop:AT16}
		Let $\caA$ be a full rank supersolvable real arrangement. Then,
		$$
			|{\rm Sing}_2(\caA)| + m(\caA) \geq |\caA|.
		$$
	\end{proposition}
	
	As a direct consequence, we obtain the following.
	
	\begin{proposition}\label{prop:supersolvable_no_unexpected_condition}
		Let $\caA$ be a full rank supersolvable real line arrangement such that $|{\rm Sing}_2(\caA)| = \frac{d}{2}$. Then, $\caA$ admits no unexpected curves.
	\end{proposition}
	\begin{proof}
		By Proposition \ref{prop:AT16}, we have $2m\geq d$. Hence the conclusion follows by Theorem \ref{prop:supersolvable_unexpected_condition}.
	\end{proof}
	
	\begin{remark}
		Examples of full rank supersolvable real arrangements with $|{\rm Sing}_2(\caA)| = \frac{d}{2}$ are given by the configurations of lines called \textit{B\"{o}r\"{o}czky examples}. These examples arise in the literature as sets of non-collinear points with the fewest number of ordinary lines, i.e. lines passing exactly through two points from the set. See \cite[Proposition 2.1]{GreenTao} for more details.
	\end{remark}
	
	\subsection{Polygonal arrangements}
	Now, we consider a family of arrangements {included in the list of {\it simplicial} line arrangements given by Gr\"unbaum} in \cite{Gru09}. 
	
	\subsection*{Construction.} Consider a regular polygon with $N \geq 3$ edges. We construct the following arrangement:
	\begin{enumerate}
		\item $e_i, ~i = 1,\ldots,N$: the lines corresponding to the sides of the $N$-gone;
		\item $m_i, ~i = 1,\ldots,N$: the lines corresponding to symmetry axes of the $N$-gone;
		\item $\ell_{\infty}$: the line at infinity.
	\end{enumerate}
	\begin{definition}
		\label{def:n-gonal arrangements}
		The line arrangement $\polyA_N = \{e_1,\ldots,e_N,m_1,\ldots,m_N\}$ is called the {\bf $N$-gonal arrangement}. The line arrangement $\overline{\polyA_{N}} = \polyA_N \cup \{\ell_\infty\}$ is called the {\bf complete $N$-gonal arrangment}.
	\end{definition}
	
	\begin{remark}
		Note that in the literature (see e.g. \cite{Gru09}) the arrangements $\polyA_N$ are denoted by $A(2N,1)$, while $\overline{\polyA_N}$ as $A(2N+1,1)$. These special configurations of lines are {\it simplicial arrangements}, i.e., all cells are triangles, and appear often as examples or counterexamples to various combinatorial problems.
	\end{remark}
	
	\begin{example}
		In the next figures, we describe the construction of the arrangements $\polyA_{4}$ and $\overline{\polyA_4}$. Note that the latter is precisely the arrangement considered in Example \ref{example:Brian_coordinates}.
	\end{example}
	
		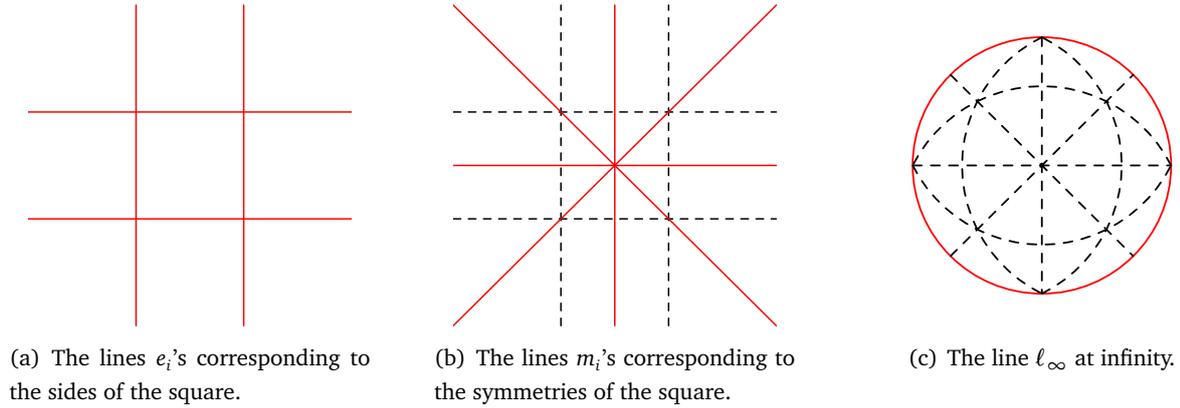
\begin{figure}[h]
			\begin{center}
				\subfigure[The lines $e_i$'s corresponding to the sides of the square.]{
					\resizebox{!}{0.25\textwidth}{
						\begin{tikzpicture}[line cap=round,line join=round,x=1.0cm,y=1.0cm]
						\clip(-1.,-1.) rectangle (2.,2.);
						\draw [color=red,domain=-1.:2.] plot(\x,{(-0.-0.*\x)/1.});
						\draw [color=red] (0.,-1.) -- (0.,2.);
						\draw [color=red] (1.,-1.) -- (1.,2.);
						\draw [color=red,domain=-1.:2.] plot(\x,{(--1.-0.*\x)/1.});
						\end{tikzpicture}
					}
				}
				\qquad
				\subfigure[The lines $m_i$'s corresponding to the symmetries of the square.]{
					\resizebox{!}{0.25\textwidth}{
						\begin{tikzpicture}[line cap=round,line join=round,x=1.0cm,y=1.0cm]
						\clip(-1.,-1.) rectangle (2.,2.);
						\draw [dash pattern=on 2pt off 2pt,domain=-1.:2.] plot(\x,{(-0.-0.*\x)/1.});
						\draw [dash pattern=on 2pt off 2pt] (0.,-1.) -- (0.,2.);
						\draw [dash pattern=on 2pt off 2pt] (1.,-1.) -- (1.,2.);
						\draw [dash pattern=on 2pt off 2pt,domain=-1.:2.] plot(\x,{(--1.-0.*\x)/1.});
						\draw [color=red,domain=-1.:2.] plot(\x,{(-0.--1.*\x)/1.});
						\draw [color=red,domain=-1.:2.] plot(\x,{(--1.-1.*\x)/1.});
						\draw [color=red,domain=-1.:2.] plot(\x,{(--0.5-0.*\x)/1.});
						\draw [color=red] (0.5,-1.) -- (0.5,2.);
						\end{tikzpicture}
					}
				}
				\qquad
				\subfigure[The line $\ell_\infty$ at infinity.]{
					\resizebox{!}{0.25\textwidth}{
						\begin{tikzpicture}[line cap=round,line join=round,x=1.0cm,y=1.0cm]
						\clip(-0.75,-0.75) rectangle (1.75,1.75);
						\draw [color=red] (0.5,0.5) circle (1.cm);
						\draw [dash pattern=on 2pt off 2pt] (-0.20710678118654754,1.2071067811865475)-- (1.2071067811865475,-0.20710678118654746);
						\draw [dash pattern=on 2pt off 2pt] (-0.20710678118654754,-0.20710678118654754)-- (1.2071067811865475,1.2071067811865475);
						\draw [dash pattern=on 2pt off 2pt] (0.5,1.5)-- (0.5,-0.5);
						\draw [dash pattern=on 2pt off 2pt] (-0.5,0.5)-- (1.5,0.5);
						\draw [shift={(0.,0.5)},dash pattern=on 2pt off 2pt]  plot[domain=-1.1071487177940904:1.1071487177940904,variable=\t]({1.*1.118033988749895*cos(\t r)+0.*1.118033988749895*sin(\t r)},{0.*1.118033988749895*cos(\t r)+1.*1.118033988749895*sin(\t r)});
						\draw [shift={(1.,0.5)},dash pattern=on 2pt off 2pt]  plot[domain=2.0344439357957027:4.2487413713838835,variable=\t]({1.*1.118033988749895*cos(\t r)+0.*1.118033988749895*sin(\t r)},{0.*1.118033988749895*cos(\t r)+1.*1.118033988749895*sin(\t r)});
						\draw [shift={(0.5,0.)},dash pattern=on 2pt off 2pt]  plot[domain=0.4636476090008061:2.677945044588987,variable=\t]({1.*1.118033988749895*cos(\t r)+0.*1.118033988749895*sin(\t r)},{0.*1.118033988749895*cos(\t r)+1.*1.118033988749895*sin(\t r)});
						\draw [shift={(0.5,1.)},dash pattern=on 2pt off 2pt]  plot[domain=3.6052402625905993:5.81953769817878,variable=\t]({1.*1.118033988749895*cos(\t r)+0.*1.118033988749895*sin(\t r)},{0.*1.118033988749895*cos(\t r)+1.*1.118033988749895*sin(\t r)});
						\end{tikzpicture}
					}
				}
			\end{center}
			\caption{Construction of $\polyA_4$ and $\overline{\polyA_4}$.}
			\label{fig:square_arr}
		\end{figure}
	
\begin{theorem}\label{thm:polygonal_unexpected}
		Let $N>2$ be an integer. Then $\polyA_{N}$ is always supersolvable, and $\overline{\polyA_{N}}$ is supersolvable if and only if $N$ is even. Moreover, $\polyA_{N}$ never admits an unexpected curve, but if $N$ is even, then $\overline{\polyA_{N}}$ admits a unique unexpected curve and its degree is $N$.
	\end{theorem}
	\begin{proof}
		{The line arrangements $\polyA_{N}$ is always supersolvable because all the singular points lie on a symmetry line; hence, the barycenter is a modular point. In $\overline{\polyA_N}$ we are adding the line at infinity; hence:
		\begin{enumerate}
			\item if $N$ is even, every line corresponding to an edge is ``parallel" to some symmetry line, i.e., they meet on the line at infinity; therefore, any singular point at infinity still lie on a symmetry line and the barycenter is still a modular point;
			\item if $N$ is odd, the line corresponding to the edges are not parallel to any symmetry line; therefore, the singular points obtained as intersection of the edge lines and the line at infinity are not connected to the barycenter of the polygon which is no longer a modular point.
		\end{enumerate}}
		By construction, the number of lines is $2N$ for $\polyA_{N}$ and $2N+1$ for $\overline{\polyA_{N}}$. Also, $m(\polyA_{N}) = m(\overline{\polyA_{N}}) = N$. Then, our claim follows directly from Theorem \ref{prop:supersolvable_unexpected_condition}. {Moreover, for $N$ even, we have that the splitting type of $\overline{\polyA_N}$ is $(N-1,N+1)$; therefore, by Theorem \ref{thm:unexpected_curve_condition_2} and Proposition \ref{prop:unexpected_curve_unique} there is a unique unexpected curve of degree $N$.}
	\end{proof}
	
	\begin{remark}
		Theorem \ref{thm:polygonal_unexpected} generalizes the case described in \cite[Example 4.1.10]{Har17} which corresponds, in our notation, to the configuration dual to $\overline{\polyA_4}$ for which we have an unexpected quartic. 
	\end{remark}
	
	{
	\begin{remark}
		Although $\bar{P_N}$ is not supersolvable when $N>2$ is odd, and thus Theorem \ref{prop:supersolvable_unexpected_condition} does not apply, computer experiments for low odd values of $N$ show that $\bar{P_N}$ has no unexpected curves.
	\end{remark}
	}
	
	\subsection{Tic-tac-toe arrangements}
	Here, we consider another family of line arrangements which generalizes \cite[Example 6.14]{CHMN17}.
	
	\subsection*{Construction.}
	A {\bf tic-tac-toe arrangement of type $(k,j)$}, denoted ${\tic}_{k}^{j}$, is the arrangement defined by:
	\begin{enumerate}
		\item[1.] $v_i$, $i = -k,\ldots,k$: vertical lines $x = kz$;
		\item[2.] $h_i$, $i = -k,\ldots,k$: horizontal lines $y = kz$;
		\item[3.] $d_i$, $i = -j,\ldots,j$: the diagonals $x - y + jz = 0$;
		\item[4.] $e_i$, $i = -j,\ldots,j$: the anti-diagonals $x + y + jz = 0$.
	\end{enumerate}
	\begin{remark}
		By symmetry, we may always assume that $k \geq j$. Indeed, thinking in the real projective plane, up to a $45^\circ$-rotation, we have that $\tic_k^j$ coincides with $\tic_j^k$. Moreover, we observe that the tic-tac-toe arrangement $\tic_1^0$ coincides with the square arrangement $\polyA_4$ (see Figure \ref{fig:square_arr}(b)), while, for $k > 1$, tic-tac-toe arrangements cannot be viewed as polygonal arrangements.
	\end{remark}
	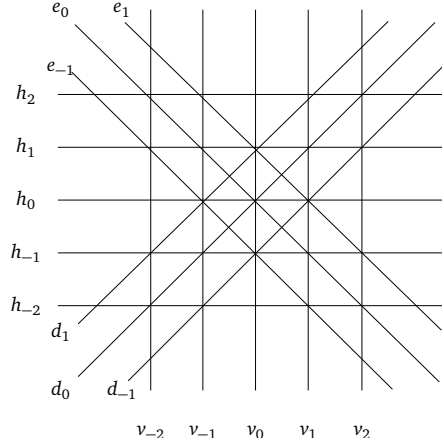
\begin{figure}[h]
		\begin{center}
			\resizebox{!}{0.35\textwidth}{
				\begin{tikzpicture}[line cap=round,line join=round,x=1.0cm,y=1.0cm]
				\clip(0.3,-4.5) rectangle (8.5,4.);
				\draw  (3.04,3.6)-- (3.04,-3.6); 
				\draw (3.04,-4.4) node {$v_{-2}$};
				\draw  (4.02,3.6)-- (4.02,-3.6);
				\draw (4.02,-4.4) node {$v_{-1}$};
				\draw (5.01,3.6)-- (5.01,-3.6);
				\draw (5.01,-4.4) node {$v_{0}$};
				\draw (6.,3.6)-- (6.,-3.6);
				\draw (6.,-4.4) node {$v_{1}$};
				\draw  (7.01,3.6)-- (7.01,-3.6);
				\draw (7.01,-4.4) node {$v_{2}$};
				\draw  (1.3,2.)-- (9.,2.);
				\draw (0.7,2.) node {$h_{2}$};
				\draw (1.3,1.)-- (9.,1.);
				\draw (0.7,1.) node {$h_{1}$};
				\draw  (1.3,0.)-- (9.,0.);
				\draw (0.7,0.) node {$h_{0}$};
				\draw  (1.3,-1.)-- (9.,-1.);
				\draw (0.7,-1.) node {$h_{-1}$};
				\draw  (1.3,-2.)-- (9.,-2.);
				\draw (0.7,-2.) node {$h_{-2}$};
				\draw (1.68,-3.34)-- (8.38,3.36);
				\draw (1.35,-3.6) node {$d_{0}$};
				\draw (1.68,-2.34)-- (7.5,3.38);
				\draw (1.35,-2.5) node {$d_{1}$};
				\draw  (2.62,-3.42)-- (8.52,2.52);
				\draw (2.5,-3.6) node {$d_{-1}$};
				\draw  (1.62,3.32)-- (8.46,-3.44);
				\draw (1.35,3.6) node {$e_{0}$};
				\draw (2.56,3.36)-- (8.5,-2.46);
				\draw (2.5,3.6) node {$e_{1}$};
				\draw (1.56,2.42)-- (7.58,-3.58);
				\draw (1.35,2.5) node {$e_{-1}$};
				\end{tikzpicture}
			}
			\caption{The tic-tac-toe arrangement of type $(2,1)$.}
		\end{center}
	\end{figure}
	Similarly as above, we denote by $\overline{\tic}_{k}^{j}$ the {\it complete} tic-tac-toe arrangement of type $(k,j)$ obtained by adding also the line at infinity. In \cite[Example 6.14]{CHMN17}, the authors observed that the splitting type of the complete tic-tac-toe arrangement of type $(k,0)$ is $(2k+1,2k+3)$ and they obtained the following.
	\begin{proposition}{\rm \cite[Proposition 6.15]{CHMN17}}
		The tic-tac-toe arrangement $\overline{\tic}_k^0$ of type $(k,0)$ admits a unique and irreducible unexpected curve of degree $2k+2$.
	\end{proposition}
	
	We may observe that $\overline{\tic}_k^0$ is supersolvable and, in particular, free. By the Addition-Deletion Theorem, we can compute the splitting type of $\overline{\tic}_k^1$ which, in particular, remains free. Therefore, we may inductively use the Addition-Deletion Theorem to compute the splitting type of $\overline{\tic}_k^j$, as we show in the following.
	\begin{lemma}\label{lemma:tictactoe_splitting}
		Let $k,j$ be positive integers with $k \geq j$. Then, the tic-tac-toe arrangement $\overline{\tic}_k^j$ is free with splitting type equal to $(2k+1+2j,2k+3+2j)$.
	\end{lemma}
	\begin{proof}
		For any $k$ and $j = 0$, we know that the claim holds by \cite[Example 6.14]{CHMN17}. We proceed now by induction on $j$. Assume that $\overline{\tic}_k^j$ is free and has splitting type $(2k+1+2j,2k+3+2j)$. We want to add the lines $d_{j+1},d_{-j-1},e_{j+1},e_{-j-1}$ and use the Addition-Deletion Theorem four times to prove the claim for $\overline{\tic}_k^{j+1}$. First, we need to compute the intersection between the diagonal $d_{j+1}$ and $\overline{\tic}_k^j$. This is:
		\begin{equation}\label{eq:counting_inter}
		|\overline{\tic}_k^j \cap d_{j+1}| = c_1 - c_2 + c_3 = [2(2k+1) + 1] - [ 2k - j] + [j+1] = 2k+2j+4;
		\end{equation}
		where:
		\begin{enumerate}
			\item[i.] $c_1$: the number of vertical and horizontal lines in $\overline{\tic}_k^j$ plus the line at infinity;
			\item[ii.] $c_2$: the number of points of the type $v_\alpha \cap h_\beta$ lying on $d_{j+1}$, i.e., the number of points of intersection in $\overline{\tic}_k^j \cap d_{j+1}$ that are counted twice by $c_1$;
			\item[iii.] $c_3$: then number of diagonals $e_i$'s intersecting $d_{j+1}$ in points not of type $v_\alpha \cap h_\beta$, i.e., the remaining points in $\overline{\tic}_k^j \cap d_{j+1}$ not yet counted by $c_1$.
		\end{enumerate}
		Then, by theAddition-Deletion Theorem, $\tic' = \overline{\tic}_k^j \cup \{d_{j+1}\}$ is free and has splitting type $(2k+2+2j,2k+3+2j)$. Now, since $d_{-j-1}\cap d_{j+1} = d_{-j-1}\cap d_0$, we have that the cardinality of the intersection $\tic' \cap d_{-j-1}$ is the same as counted in \eqref{eq:counting_inter}. Therefore, by the Addition-Deletion Theorem, $\tic'' = \tic' \cup \{d_{-j-1}\}$ is free and has splitting type $(2k+3+2j,2k+3+2j)$. Similarly, since for any $\alpha,\beta$, $e_\alpha \cap d_\beta = v_{\alpha'} \cap h_{\beta'}$, for some $\alpha',\beta'$, we have that also the intersections $\tic'' \cap \{e_{j+1}\}$ and $(\tic'' \cup \{e_{j+1}\}) \cap \{e_{-j-1}\}$ have the same cardinality as counted in \eqref{eq:counting_inter}. Again, by the Addition-Deletion Theorem, we have that the line arrangement
		$$
		\overline{\tic}_k^j \cup \{d_{j+1},d_{-j-1},e_{j+1},e_{-j-1}\} = \overline{\tic}_k^{j+1},
		$$
		is free and has splitting type $(2k+3+2j,2k+5+2j) = (2k+1+2(j+1),2k+3+2(j+1))$.
	\end{proof}
	\begin{theorem}
		The complete arrangement $\overline{\tic}^j_k$ {admits a unique} unexpected curve of degree $2(k+j+1)$.
	\end{theorem}
	\begin{proof}
		Observe that $m(\overline{\tic}_k^j)=2k+1$, that is the multiplicity of one of the points at infinity, e.g. the direction of the vertical lines. Hence in the dual configuration there are no more than $2k+1$ collinear points. {Then, we can use Lemma \ref{lemma:tictactoe_splitting} and conclude by Theorem \ref{thm:unexpected_curve_condition_2} and Proposition \ref{prop:unexpected_curve_unique}.}
	\end{proof}
%	
%	\subsection{Nearly free and nearly supersolvable arrangements}
%	In the previous sections, we only considered free line arrangements. Here we want to focus on some particular families of arrangements that are not free and appeared recently in the literature.

	\section{Other examples}\label{appendix:further examples}
	In this section, we exhibit other examples of unexpected curves arising from special line arrangements. 
	
	{
	\subsection{Adding lines to polygonal arrangements}
	First, we construct them by using ideas from the previous sections. We may notice that tic-tac-toe arrangements are constructed from the square arrangement ${\polyA}_4$ by adding lines parallel to the ones of ${\polyA}_4$. Hence, we try to proceed in a similar way by starting from polygonal arrangements $\polyA_N$, with $N$ even. Unfortunately, this procedure is not successful in the sense that we can use Addition-Deletion Theorem only in a very few cases, as we are going to explain, but in general we do not know how to efficiently compute the splitting type of these line arrangements, since they are not supersolvable (hence, we cannot use Lemma \ref{lem:splitting_supersolvable}) and we cannot apply Addition-Deletion Theorem.}
	
\begin{example}
		\label{ex:A(13,1) another ex}
		Consider the set of lines from the arrangement $\overline{\polyA_{6}}$ and take the points $P_1,\ldots, P_6$ as in Figure \ref{fig:A(13,1)andPoints}.
		From the proof of Theorem \ref{thm:polygonal_unexpected} we know that the splitting type for $\overline{\polyA_{6}}$ is $(5,7)$. Now, we construct a series of examples for which some of the dual configuration of the points give an unexpected curve. 
\begin{figure}[htb]
\begin{center}
					\begin{tikzpicture}[line cap=round,line join=round,>=triangle 45,x=1.0cm,y=1.0cm, scale=3]
				\clip(-0.83,-0.68) rectangle (0.88,0.69);
				\draw [domain=-0.83:0.88] plot(\x,{(-1-0*\x)/4});
				\draw [domain=-0.83:0.88] plot(\x,{(-1--3*\x)/2});
				\draw [domain=-0.83:0.88] plot(\x,{(-1--3*\x)/-2});
				\draw [domain=-0.83:0.88] plot(\x,{(-1-0*\x)/-4});
				\draw [domain=-0.83:0.88] plot(\x,{(-1-3*\x)/-2});
				\draw [domain=-0.83:0.88] plot(\x,{(-1-3*\x)/2});
				\draw [domain=-0.83:0.88] plot(\x,{(-0--3*\x)/2});
				\draw [domain=-0.83:0.88] plot(\x,{(-0--1*\x)/2});
				\draw [domain=-0.83:0.88] plot(\x,{(-0-0*\x)/1});
				\draw [domain=-0.83:0.88] plot(\x,{(-0-1*\x)/2});
				\draw [domain=-0.83:0.88] plot(\x,{(-0-3*\x)/2});
				\draw (0,-0.68) -- (0,0.69);
				\begin{scriptsize}
				\fill [color=black] (-0.5,-0.25) circle (0.5pt);
				\draw[color=black] (-0.42,-0.31) node {$P_1$};
				\fill [color=black] (-0.5,0.25) circle (0.5pt);
				\draw[color=black] (-0.4,0.36) node {$P_2$};
				\fill [color=black] (0,0.5) circle (0.5pt);
				\draw[color=black] (0.08,0.53) node {$P_3$};
				\fill [color=black] (0.5,0.25) circle (0.5pt);
				\draw[color=black] (0.49,0.33) node {$P_4$};
				\fill [color=black] (0.5,-0.25) circle (0.5pt);
				\draw[color=black] (0.57,-0.19) node {$P_5$};
				\fill [color=black] (0,-0.5) circle (0.5pt);
				\draw[color=black] (0.11,-0.5) node {$P_6$};
				\end{scriptsize}
				\end{tikzpicture}
	\caption{Configuration $\overline{\polyA_{6}}$. The line at infinity is not shown.}
	\label{fig:A(13,1)andPoints}
\end{center}
\end{figure}
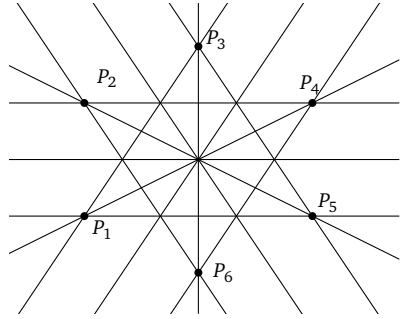

			We add, step by step, the lines $\ell_1 := \overline{P_1P_2}, \ell_2:=\overline{P_2P_3},\ldots,\ell_6:=\overline{P_6P_1}$ (blue dotted lines in Figure \ref{fig:A(13,1)andLines}). Denote $\caB_0 := \overline{\polyA_6}$ and $\caB_i := \caB_{i-1}\cup\{\ell_i\}$. By Theorem \ref{thm:addition-deletion}, since $$\Big|\bigcup_{\ell \in \caB_{i-1}} \ell \cap \ell_{i}\Big| = 8, \text{ for } i = 1,\ldots,6,$$ the splitting types of the arrangements $\caB_i$'s are
			
$$
		\begin{matrix}
	\overline{\polyA_6}  \\ (5,7)
	\end{matrix} ~~\longrightarrow~~
		\begin{matrix}
	\caB_1 = \overline{\polyA_6} \cup \{\ell_1\}  \\ (6,7)
	\end{matrix} ~~\longrightarrow~~
		\begin{matrix}
	\caB_2 = \overline{\polyA_6} \cup\{\ell_1,\ell_2\}  \\ (7,7)
	\end{matrix} ~~\longrightarrow~~
		\ldots~~ \longrightarrow~~
		\begin{matrix}
	\caB_6 = \overline{\polyA_6} \cup\{\ell_1,\ldots,\ell_6\}  \\ (11,7)
	\end{matrix} 
$$
By Theorem \ref{thm:unexpected_curve_condition_2}, we have that the line arrangements $\caB_4,\caB_5$ and $\caB_6$ admit unexpected curves of degrees $8,9$ and $10$. We continue by adding lines passing through the points $P_1,\ldots,P_6$, as indicated in Figure \ref{fig:A(13,1)andLines} (dashed-dotted red lines). Denote by $\ell_i'$ the new line passing through $P_i$, respectively, and denote the line arrangements $\caB_0' := \caB_6$ and $\caB'_i := \caB'_{i-1}\cup \{\ell'_i\}$. Since 
$$\Big|\bigcup_{\ell\in\caB'_{i-1}}\ell \cap \ell'_i\Big| = 12\text{, for i = 1,\ldots,6},$$ by Theorem \ref{thm:addition-deletion}, the splitting types of the arrangements $\caB'_i$'s are
	$$
		\begin{matrix}
	\caB_6  \\ (11,7)
	\end{matrix} ~~\longrightarrow~~
		\begin{matrix}
	\caB'_1 = \caB_6 \cup \{\ell'_1\}  \\ (11,8)
	\end{matrix} ~~\longrightarrow~~
		\begin{matrix}
	\caB'_2 = \caB_6 \cup\{\ell'_1,\ell'_2\}  \\ (11,9)
	\end{matrix} ~~\longrightarrow~~
		\ldots~~ \longrightarrow~~
		\begin{matrix}
	\caB'_6 = \caB_6 \cup\{\ell'_1,\ldots,\ell'_6\}  \\ (11,13)
	\end{matrix} 
$$
			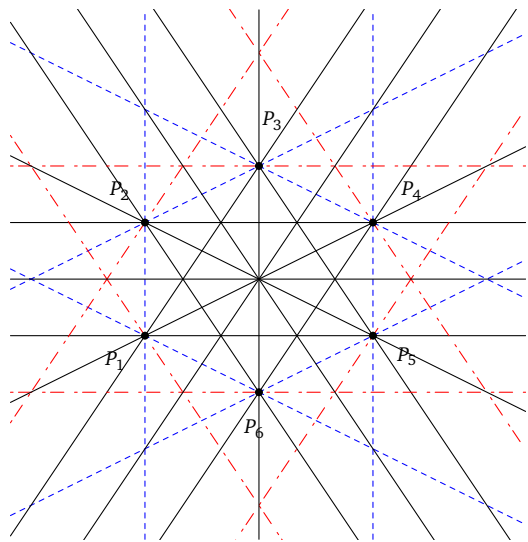
\begin{figure}[H]
				\centering
				\begin{tikzpicture}[line cap=round,line join=round,>=triangle 45,x=1.0cm,y=1.0cm, scale=3]
				\clip(-1.09,-1.15) rectangle (1.17,1.19);
				\draw [dash pattern=on 2pt off 2pt, color = blue] (-0.5,-1.15) -- (-0.5,1.19);
				\draw [dash pattern=on 2pt off 2pt,domain=-1.09:1.17, color = blue] plot(\x,{(-1-1*\x)/2});
				\draw [dash pattern=on 2pt off 2pt,domain=-1.09:1.17, color = blue] plot(\x,{(-1--1*\x)/2});
				\draw [dash pattern=on 2pt off 2pt, color = blue] (0.5,-1.15) -- (0.5,1.19);
				\draw [dash pattern=on 2pt off 2pt,domain=-1.09:1.17, color = blue] plot(\x,{(-1--1*\x)/-2});
				\draw [dash pattern=on 2pt off 2pt,domain=-1.09:1.17, color = blue] plot(\x,{(-1-1*\x)/-2});
				\draw [domain=-1.09:1.17] plot(\x,{(-1-0*\x)/4});
				\draw [domain=-1.09:1.17] plot(\x,{(-1--3*\x)/2});
				\draw [domain=-1.09:1.17] plot(\x,{(-1--3*\x)/-2});
				\draw [domain=-1.09:1.17] plot(\x,{(-1-0*\x)/-4});
				\draw [domain=-1.09:1.17] plot(\x,{(-1-3*\x)/-2});
				\draw [domain=-1.09:1.17] plot(\x,{(-1-3*\x)/2});
				\draw [dash pattern=on 1pt off 2pt on 5pt off 4pt,color=black,domain=-1.09:1.17, color = red] plot(\x,{(-1-0*\x)/2});
				\draw [dash pattern=on 1pt off 2pt on 5pt off 4pt, color=black,domain=-1.09:1.17, color = red] plot(\x,{(-1--1.5*\x)/1});
				\draw [dash pattern=on 1pt off 2pt on 5pt off 4pt, color=black,domain=-1.09:1.17, color = red] plot(\x,{(-1--1.5*\x)/-1});
				\draw [dash pattern=on 1pt off 2pt on 5pt off 4pt, color=black,domain=-1.09:1.17, color = red] plot(\x,{(-1-0*\x)/-2});
				\draw [dash pattern=on 1pt off 2pt on 5pt off 4pt, color=black,domain=-1.09:1.17, color = red] plot(\x,{(-1-1.5*\x)/-1});
				\draw [dash pattern=on 1pt off 2pt on 5pt off 4pt, color=black,domain=-1.09:1.17, color = red] plot(\x,{(-1-1.5*\x)/1});
				\draw [domain=-1.09:1.17] plot(\x,{(-0--3*\x)/2});
				\draw [domain=-1.09:1.17] plot(\x,{(-0--1*\x)/2});
				\draw [domain=-1.09:1.17] plot(\x,{(-0-0*\x)/1});
				\draw [domain=-1.09:1.17] plot(\x,{(-0-1*\x)/2});
				\draw [domain=-1.09:1.17] plot(\x,{(-0-3*\x)/2});
				\draw (0,-1.15) -- (0,1.19);
				\begin{scriptsize}
				\fill [color=black] (-0.5,-0.25) circle (0.5pt);
				\draw[color=black] (-0.63,-0.36) node {$P_1$};
				\fill [color=black] (-0.5,0.25) circle (0.5pt);
				\draw[color=black] (-0.61,0.39) node {$P_2$};
				\fill [color=black] (0,0.5) circle (0.5pt);
				\draw[color=black] (0.06,0.7) node {$P_3$};
				\fill [color=black] (0.5,0.25) circle (0.5pt);
				\draw[color=black] (0.67,0.39) node {$P_4$};
				\fill [color=black] (0.5,-0.25) circle (0.5pt);
				\draw[color=black] (0.65,-0.34) node {$P_5$};
				\fill [color=black] (0,-0.5) circle (0.5pt);
				\draw[color=black] (-0.02,-0.66) node {$P_6$};
				\end{scriptsize}
				\end{tikzpicture}
				\caption{Dashed blue lines $\ell_1,\ldots,\ell_6$ added in the first step; dash-dotted red lines $\ell'_1,\ldots,\ell'_6$ in the second.}
				\label{fig:A(13,1)andLines}
			\end{figure}
			We obtain three new arrangements $\caB'_1,\caB'_2$ and $\caB'_6$ which admit unexpected curves of degree $9,10$ and $12$. Moreover, we may check that the line arrangement $\caB'_6$ constructed in Figure \ref{fig:A(13,1)andLines} is dual to the configuration of points given by $\Sing(\overline{\polyA_{6}})$. This procedure of adding lines can be repeated two more times. As indicated in Figure \ref{fig:38linesFromA(13,1)}, we first add the $6$ blue dashed lines, $m_1,\ldots,m_6$, and then $6$ red dash-dotted lines $m'_1,\ldots,m'_6$.
			
			In this process, by Theorem \ref{thm:addition-deletion}, we get the following series of exponents
$$
		\begin{matrix}
	\caB'_6  \\ (11,13)
	\end{matrix} ~~\longrightarrow~~
		\begin{matrix}
	\caB'_6 \cup \{m_1\}  \\ (12,13)
	\end{matrix} ~~\longrightarrow~~
		\begin{matrix}
	\caB'_6 \cup\{m_1,m_2\}  \\ (13,13)
	\end{matrix} ~~\longrightarrow~~
		\ldots~~ \longrightarrow~~
		\begin{matrix}
	\caB''_6 := \caB'_6 \cup\{m_1,\ldots,m_6\}  \\ (17,13)
	\end{matrix} 
$$
$$
		\begin{matrix}
	\caB'_6  \\ (17,13)
	\end{matrix} ~~\longrightarrow~~
		\begin{matrix}
	\caB''_6 \cup \{m'_1\}  \\ (17,14)
	\end{matrix} ~~\longrightarrow~~
		\begin{matrix}
	\caB''_6 \cup\{m'_1,m'_2\}  \\ (17,15)
	\end{matrix} ~~\longrightarrow~~
		\ldots~~ \longrightarrow~~
		\begin{matrix}
	\caB''_6 \cup\{m'_1,\ldots,m'_6\}  \\ (17,19)
	\end{matrix} 
$$
from which, by Theorem \ref{thm:unexpected_curve_condition_2}, we find new examples of unexpected curves.
			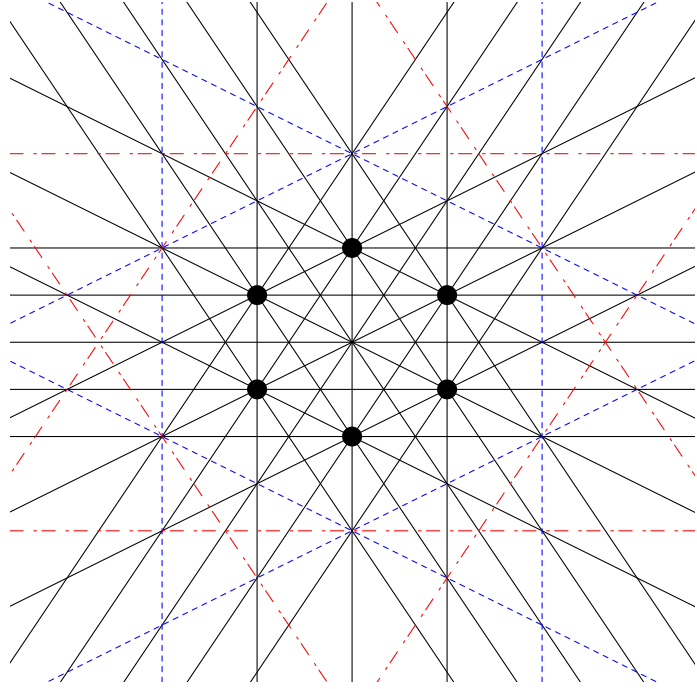
\begin{figure}[H]
			\centering
			\begin{tikzpicture}[line cap=round,line join=round,>=triangle 45,x=1.0cm,y=1.0cm, scale=2.5]
			\clip(-1.8,-1.8) rectangle (1.8,1.8);
			\draw (-0.5,-2.31) -- (-0.5,2.31);
			\draw [domain=-2.35:2.36] plot(\x,{(-1-1*\x)/2});
			\draw [domain=-2.35:2.36] plot(\x,{(-1--1*\x)/2});
			\draw (0.5,-2.31) -- (0.5,2.31);
			\draw [domain=-2.35:2.36] plot(\x,{(-1--1*\x)/-2});
			\draw [domain=-2.35:2.36] plot(\x,{(-1-1*\x)/-2});
			\draw [domain=-2.35:2.36] plot(\x,{(-1-0*\x)/4});
			\draw [domain=-2.35:2.36] plot(\x,{(-1--3*\x)/2});
			\draw [domain=-2.35:2.36] plot(\x,{(-1--3*\x)/-2});
			\draw [domain=-2.35:2.36] plot(\x,{(-1-0*\x)/-4});
			\draw [domain=-2.35:2.36] plot(\x,{(-1-3*\x)/-2});
			\draw [domain=-2.35:2.36] plot(\x,{(-1-3*\x)/2});
			\draw [domain=-2.35:2.36] plot(\x,{(-1-0*\x)/2});
			\draw [domain=-2.35:2.36] plot(\x,{(-1--1.5*\x)/1});
			\draw [domain=-2.35:2.36] plot(\x,{(-1--1.5*\x)/-1});
			\draw [domain=-2.35:2.36] plot(\x,{(-1-0*\x)/-2});
			\draw [domain=-2.35:2.36] plot(\x,{(-1-1.5*\x)/-1});
			\draw [domain=-2.35:2.36] plot(\x,{(-1-1.5*\x)/1});
			\draw [domain=-2.35:2.36] plot(\x,{(-0--3*\x)/2});
			\draw [domain=-2.35:2.36] plot(\x,{(-0--1*\x)/2});
			\draw [domain=-2.35:2.36] plot(\x,{(-0-0*\x)/1});
			\draw [domain=-2.35:2.36] plot(\x,{(-0-1*\x)/2});
			\draw [domain=-2.35:2.36] plot(\x,{(-0-3*\x)/2});
			\draw [color=black] (0,-2.31) -- (0,2.31);
			\draw [dash pattern=on 2pt off 2pt, color=blue] (-1,-2.31) -- (-1,2.31);
			\draw [dash pattern=on 2pt off 2pt, color=blue,domain=-2.35:2.36] plot(\x,{(--0.5--0.25*\x)/0.5});
			\draw [dash pattern=on 2pt off 2pt, color=blue,domain=-2.35:2.36] plot(\x,{(--1-0.5*\x)/1});
			\draw [dash pattern=on 2pt off 2pt, color=blue] (1,-2.31) -- (1,2.31);
			\draw [dash pattern=on 2pt off 2pt, color=blue,domain=-2.35:2.36] plot(\x,{(--0.5-0.25*\x)/-0.5});
			\draw [dash pattern=on 2pt off 2pt, color=blue,domain=-2.35:2.36] plot(\x,{(--1--0.5*\x)/-1});
			\draw [dash pattern=on 1pt off 2pt on 5pt off 4pt, color=red,domain=-2.35:2.36] plot(\x,{(-6-0*\x)/-6});
			\draw [dash pattern=on 1pt off 2pt on 5pt off 4pt, color=red,domain=-2.35:2.36] plot(\x,{(-6--4.5*\x)/-3});
			\draw [dash pattern=on 1pt off 2pt on 5pt off 4pt, color=red,domain=-2.35:2.36] plot(\x,{(-6--4.5*\x)/3});
			\draw [dash pattern=on 1pt off 2pt on 5pt off 4pt, color=red,domain=-2.35:2.36] plot(\x,{(-2-0*\x)/2});
			\draw [dash pattern=on 1pt off 2pt on 5pt off 4pt, color=red,domain=-2.35:2.36] plot(\x,{(-6-4.5*\x)/3});
			\draw [dash pattern=on 1pt off 2pt on 5pt off 4pt, color=red,domain=-2.35:2.36] plot(\x,{(-6-4.5*\x)/-3});
			\begin{scriptsize}
			\fill [color=black] (-0.5,-0.25) circle (1.5pt);
			\fill [color=black] (-0.5,0.25) circle (1.5pt);
			\fill [color=black] (0,0.5) circle (1.5pt);
			\fill [color=black] (0.5,0.25) circle (1.5pt);
			\fill [color=black] (0.5,-0.25) circle (1.5pt);
			\fill [color=black] (0,-0.5) circle (1.5pt);
			\end{scriptsize}
			\end{tikzpicture}
			\caption{Bolded points indicate the original points $P_1,\ldots,P_6$; blue dotted lines $m_1,\ldots,m_6$ are added in the first step; red dash-dotted lines $m'_1,\ldots,m'_6$ are added in the second step.}
			\label{fig:38linesFromA(13,1)}
		\end{figure}
	\end{example}
	
	\begin{example}
		\label{ex:octagon}
		We can proceed in a similar way as in Example \ref{ex:A(13,1) another ex}, but starting from configuration $\overline{\polyA_{8}}$. As before, we denote by $P_1,P_2,\ldots,P_8$ vertices of the octagon as indicated in Figure \ref{fig:A(17,1)}.
		\begin{figure}[H]
			\centering
			\includegraphics[scale=0.4]{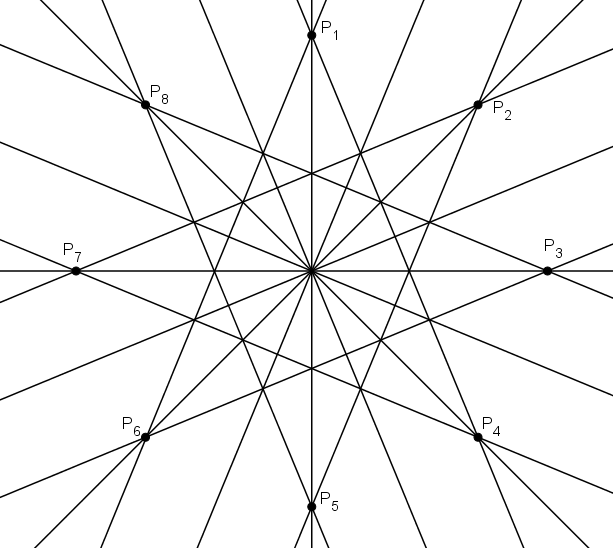}
			%here is a problem with tikz picture...
			\caption{Configuration $\overline{\polyA_{8}}$. The line at infinity is not shown.}
			\label{fig:A(17,1)}
		\end{figure}
	By Lemma \ref{lem:splitting_supersolvable}, the splitting type of $\overline{\polyA_{8}}$ is $(7,9)$. We add the lines $\ell_1 := \overline{P_1P_2}$, $\ell_2 := \overline{P_2P_3}$, $\ldots,\ell_8 := \overline{P_8P_1}$ (blue dotted lines in Figure \ref{fig:A(17,1)+8lines}). By Theorem \ref{thm:addition-deletion}, the splitting type of $\overline{\polyA_8} \cup \{\ell_1,\ldots,\ell_i\} = (7+i,9)$, for all $i = 1,\ldots,8$
	Thus, the existence of unexpected curves for $i \in \{4,5,\ldots,8\}$ is guaranteed by Theorem \ref{thm:unexpected_curve_condition_2}.
	
	This line arrangement can also be extended to new arrangements which admit unexpected curves. We add other $8$ lines: $m_1 := \overline{P_1P_3}$, $m_2 := \overline{P_2P_4}$, $\ldots,m_8 := \overline{P_8P_2}$ (red dash-dotted lines in Figure \ref{fig:A(17,1)+8lines}). By Theorem \ref{thm:addition-deletion}, the splitting type of $\overline{\polyA_8} \cup \{\ell_1,\ldots,\ell_6,m_1,\ldots,m_j\} = (15,9+j)$, for all $j = 1,\ldots,8$.
	Thus, the existence of unexpected curves for $j \in \{1,2,3,4,8\}$ is guaranteed by Theorem \ref{thm:unexpected_curve_condition_2}.
		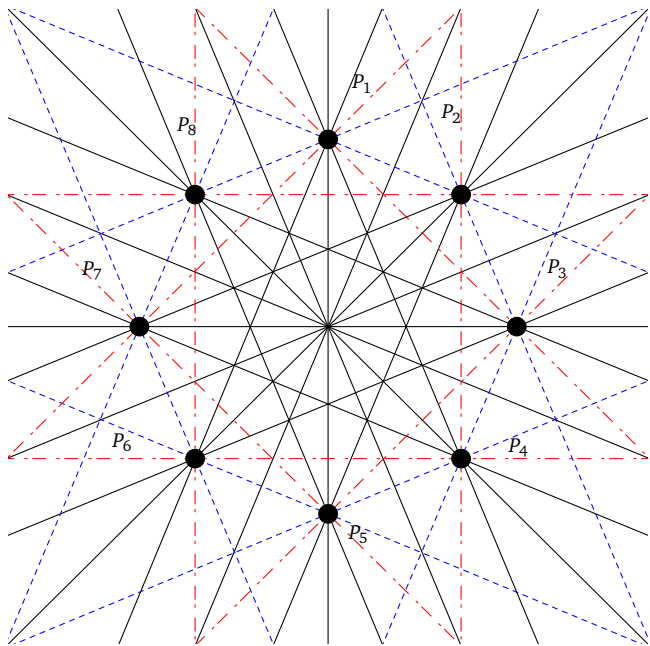
\begin{figure}[H]
			\centering
			\begin{tikzpicture}[line cap=round,line join=round,>=triangle 45,x=1.0cm,y=1.0cm, scale=3.5]
			\clip(-1.2,-1.2) rectangle (1.2,1.2);
			\draw [domain=-1.8616403900545755:1.7828870091513567] plot(\x,{(-1.--1.41*\x)/-3.41});
			\draw [domain=-1.8616403900545755:1.7828870091513567] plot(\x,{(-1.-1.41*\x)/-3.41});
			\draw [domain=-1.8616403900545755:1.7828870091513567] plot(\x,{(-1.-3.41*\x)/1.41});
			\draw [domain=-1.8616403900545755:1.7828870091513567] plot(\x,{(-1.-1.41*\x)/3.41});
			\draw [domain=-1.8616403900545755:1.7828870091513567] plot(\x,{(-1.--1.41*\x)/3.41});
			\draw [domain=-1.8616403900545755:1.7828870091513567] plot(\x,{(-1.--3.41*\x)/1.41});
			\draw [domain=-1.8616403900545755:1.7828870091513567] plot(\x,{(-1.--3.41*\x)/-1.41});
			\draw [domain=-1.8616403900545755:1.7828870091513567] plot(\x,{(-0.--1.41*\x)/3.41});
			\draw [domain=-1.8616403900545755:1.7828870091513567] plot(\x,{(-0.-1.41*\x)/3.41});
			\draw [,domain=-1.8616403900545755:1.7828870091513567] plot(\x,{(-0.-3.41*\x)/1.41});
			\draw [domain=-1.8616403900545755:1.7828870091513567] plot(\x,{(-0.-3.41*\x)/-1.41});
			\draw [domain=-1.8616403900545755:1.7828870091513567] plot(\x,{(-0.-0.*\x)/1.});
			\draw (0.,-1.9271653979677097) -- (0.,2.018157928619877);
			\draw [domain=-1.8616403900545755:1.7828870091513567] plot(\x,{(-0.-1.*\x)/1.});
			\draw [domain=-1.8616403900545755:1.7828870091513567] plot(\x,{(-0.-1.*\x)/-1.});
			\draw [domain=-1.8616403900545755:1.7828870091513567] plot(\x,{(-0.1462215320910973-0.50101966873706*\x)/-0.20601966873706007});
			\draw [dash pattern=on 2pt off 2pt, domain=-1.8616403900545755:1.7828870091513567, color = blue] plot(\x,{(--0.35460992907801414-0.2092198581560285*\x)/0.5});
			\draw [dash pattern=on 2pt off 2pt, domain=-1.8616403900545755:1.7828870091513567, color = blue] plot(\x,{(--0.3546099290780142-0.5*\x)/0.20921985815602862});
			\draw [dash pattern=on 2pt off 2pt, domain=-1.8616403900545755:1.7828870091513567, color = blue] plot(\x,{(--0.3546099290780142-0.5*\x)/-0.20921985815602862});
			\draw [dash pattern=on 2pt off 2pt, domain=-1.8616403900545755:1.7828870091513567, color = blue] plot(\x,{(--0.3546099290780142-0.20921985815602862*\x)/-0.5});
			\draw [dash pattern=on 2pt off 2pt, domain=-1.8616403900545755:1.7828870091513567, color = blue] plot(\x,{(--0.35460992907801425--0.2092198581560285*\x)/-0.5});
			\draw [dash pattern=on 2pt off 2pt, domain=-1.8616403900545755:1.7828870091513567, color = blue] plot(\x,{(--0.35460992907801425--0.5*\x)/-0.2092198581560285});
			\draw [dash pattern=on 2pt off 2pt, domain=-1.8616403900545755:1.7828870091513567, color = blue] plot(\x,{(--0.3546099290780142--0.5*\x)/0.20921985815602862});
			\draw [dash pattern=on 2pt off 2pt, domain=-1.8616403900545755:1.7828870091513567, color = blue] plot(\x,{(--0.35460992907801414--0.2092198581560285*\x)/0.5});
						\draw [dash pattern=on 1pt off 2pt on 5pt off 4pt, color = red] (0.5,-2.3888521702279593) -- (0.5,1.9621958956186345);
			\draw [dash pattern=on 1pt off 2pt on 5pt off 4pt, domain=-2.2253936045626532:2.06269717415757, color = red] plot(\x,{(-0.5029928072028572--0.7092198581560285*\x)/0.7092198581560285});
			\draw [dash pattern=on 1pt off 2pt on 5pt off 4pt, domain=-2.2253936045626532:2.06269717415757, color = red] plot(\x,{(--0.5-0.*\x)/-1.});
			\draw [dash pattern=on 1pt off 2pt on 5pt off 4pt, domain=-2.2253936045626532:2.06269717415757, color = red] plot(\x,{(--0.5029928072028572--0.7092198581560285*\x)/-0.7092198581560285});
			\draw [dash pattern=on 1pt off 2pt on 5pt off 4pt, color = red] (-0.5,-2.3888521702279593) -- (-0.5,1.9621958956186345);
			\draw [dash pattern=on 1pt off 2pt on 5pt off 4pt, domain=-2.2253936045626532:2.06269717415757, color = red] plot(\x,{(--0.5029928072028571--0.7092198581560284*\x)/0.7092198581560285});
			\draw [dash pattern=on 1pt off 2pt on 5pt off 4pt,  domain=-2.2253936045626532:2.06269717415757, color = red] plot(\x,{(--0.5-0.*\x)/1.});
			\draw [dash pattern=on 1pt off 2pt on 5pt off 4pt, domain=-2.2253936045626532:2.06269717415757, color = red] plot(\x,{(--0.5029928072028571-0.7092198581560284*\x)/0.7092198581560285});
			\begin{scriptsize}
			\draw [fill=black] (0.,0.7092198581560284) circle (1.0pt);
			\draw[color=black] (0.12850940855211876,0.9199030309704954) node {$P_1$};
			\draw [fill=black] (0.5,0.5) circle (1.0pt);
			\draw[color=black] (0.46428160655957507,0.8079789649680107) node {$P_2$};
			\draw [fill=black] (0.7092198581560285,0.) circle (1.0pt);
			\draw[color=black] (0.8630110916934295,0.2203776184549659) node {$P_3$};
			\draw [fill=black] (0.5,-0.5) circle (1.0pt);
			\draw[color=black] (0.7161107550651674,-0.4581620316850977) node {$P_4$};
			\draw [fill=black] (0.,-0.7092198581560285) circle (1.0pt);
			\draw[color=black] (0.11451890030180809,-0.7869389755673966) node {$P_5$};
			\draw [fill=black] (-0.5,-0.5) circle (1.0pt);
			\draw[color=black] (-0.7738783735929201,-0.4371762693096318) node {$P_6$};
			\draw [fill=black] (-0.7092198581560285,0.) circle (1.0pt);
			\draw[color=black] (-0.8858024395954055,0.21338236432981061) node {$P_7$};
			\draw [fill=black] (-0.5,0.5) circle (1.0pt);
			\draw[color=black] (-0.5290444792124832,0.7590121860919237) node {$P_8$};
			\end{scriptsize}
			\end{tikzpicture}
			\caption{Blue dotted lines $\ell_1,\ldots,\ell_8$ added to $\overline{\polyA_8}$.}
			\label{fig:A(17,1)+8lines}
		\end{figure}
\end{example}
	
		\begin{remark}
			Observe that the order of adding new lines in Example \ref{ex:octagon} is not {relevant because the lines we are adding meet each other in points which are also intersections with the original lines of the polygonal arrangement}.
		\end{remark}
	
	\subsection{Sporadic cases}
	It occurs that not only the dual configurations of points to the line arrangements presented in the previous sections are giving unexpected curves. There are {other} simplicial arrangements for which we have the same property. In Table \ref{tab:A(n,k) sporadic dual}, we present a list of arrangements, together with their splitting type, and with their original names coming as in \cite{Gru09}. In this list, we consider line arrangements that are {\it dual} to the configurations $A(n,k)$ described in Gr\"unbaum's paper.
	
	\begin{definition}
		Given a line arrangement $\caA$, we define its {\bf dual line arrangement}, denoted by $\caA^d$, as the line arrangement dual to the configuration of points $\Sing(\caA)$.
	\end{definition}
	
	\begin{example}
		The next three figures illustrate examples of the previous definitions and notations.
		\begin{center}
		\begin{minipage}{0.4\textwidth}
			\begin{figure}[H]
			\begin{center}
				\begin{tikzpicture}[line cap=round,line join=round,>=triangle 45,x=1.0cm,y=1.0cm,scale=0.5]
				\clip(0.88,-4.7) rectangle (10.9,4.94);
				\draw  (6,0) circle (4.47cm);
				\draw  (6.02,4.47)-- (6.04,-4.47);
				\draw  (1.53,-0.04)-- (10.47,-0.02);
				\draw [shift={(8.96,0.01)}]  plot[domain=2.15:4.13,variable=\t]({1*5.35*cos(\t r)+0*5.35*sin(\t r)},{0*5.35*cos(\t r)+1*5.35*sin(\t r)});
				\draw [shift={(2.96,-0.01)}]  plot[domain=-0.97:0.97,variable=\t]({1*5.42*cos(\t r)+0*5.42*sin(\t r)},{0*5.42*cos(\t r)+1*5.42*sin(\t r)});
				\draw [shift={(6.01,-2.95)}]  plot[domain=0.58:2.57,variable=\t]({1*5.34*cos(\t r)+0*5.34*sin(\t r)},{0*5.34*cos(\t r)+1*5.34*sin(\t r)});
				\draw [shift={(5.99,3.06)}]  plot[domain=3.75:5.68,variable=\t]({1*5.44*cos(\t r)+0*5.44*sin(\t r)},{0*5.44*cos(\t r)+1*5.44*sin(\t r)});
				\draw (9.19,3.13)-- (2.85,-3.17);
				\draw (2.82,3.14)-- (9.17,-3.15);
				\end{tikzpicture}
				\caption{Line arrangements $\overline{\polyA_{4}}$. \vspace{1.2cm}}
				\label{fig:P4}
				\end{center}
			\end{figure}
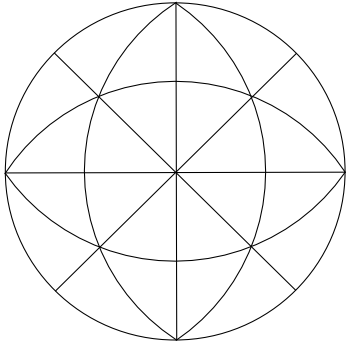
		\end{minipage}~~~
		\begin{minipage}{0.4\textwidth}
			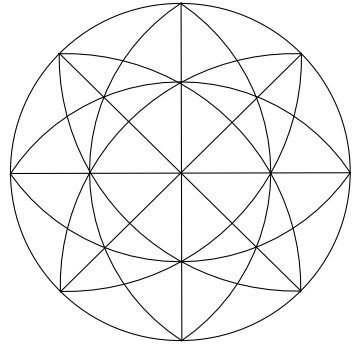
\begin{figure}[H]
			\begin{center}
				\begin{tikzpicture}[line cap=round,line join=round,>=triangle 45,x=1.0cm,y=1.0cm,scale=0.5]
				\clip(0.88,-4.7) rectangle (10.9,4.94);
				\draw (6,0) circle (4.47cm);
				\draw (6.02,4.47)-- (6.04,-4.47);
				\draw (1.53,-0.04)-- (10.47,-0.02);
				\draw [shift={(8.96,0.01)}]  plot[domain=2.15:4.13,variable=\t]({1*5.35*cos(\t r)+0*5.35*sin(\t r)},{0*5.35*cos(\t r)+1*5.35*sin(\t r)});
				\draw [shift={(2.96,-0.01)}]  plot[domain=-0.97:0.97,variable=\t]({1*5.42*cos(\t r)+0*5.42*sin(\t r)},{0*5.42*cos(\t r)+1*5.42*sin(\t r)});
				\draw [shift={(6.01,-2.95)}]  plot[domain=0.58:2.57,variable=\t]({1*5.34*cos(\t r)+0*5.34*sin(\t r)},{0*5.34*cos(\t r)+1*5.34*sin(\t r)});
				\draw [shift={(5.99,3.06)}]  plot[domain=3.75:5.68,variable=\t]({1*5.44*cos(\t r)+0*5.44*sin(\t r)},{0*5.44*cos(\t r)+1*5.44*sin(\t r)});
				\draw  (9.19,3.13)-- (2.85,-3.17);
				\draw (2.82,3.14)-- (9.17,-3.15);
				\draw [shift={(8.99,-3.01)}] plot[domain=1.54:3.17,variable=\t]({1*6.15*cos(\t r)+0*6.15*sin(\t r)},{0*6.15*cos(\t r)+1*6.15*sin(\t r)});
				\draw [shift={(2.99,3.03)}] plot[domain=-1.59:0.02,variable=\t]({1*6.2*cos(\t r)+0*6.2*sin(\t r)},{0*6.2*cos(\t r)+1*6.2*sin(\t r)});
				\draw [shift={(9.05,3.08)}] plot[domain=3.13:4.73,variable=\t]({1*6.24*cos(\t r)+0*6.24*sin(\t r)},{0*6.24*cos(\t r)+1*6.24*sin(\t r)});
				\draw [shift={(2.97,-3.06)}] plot[domain=-0.02:1.6,variable=\t]({1*6.2*cos(\t r)+0*6.2*sin(\t r)},{0*6.2*cos(\t r)+1*6.2*sin(\t r)});
				\end{tikzpicture}
				\caption{Line arrangements $\overline{\polyA_{4}^d}$. \vspace{1.2cm}}
				\label{fig:P4 dual}
				\end{center}
			\end{figure}
		\end{minipage}\\
		\begin{minipage}{0.75\textwidth}
			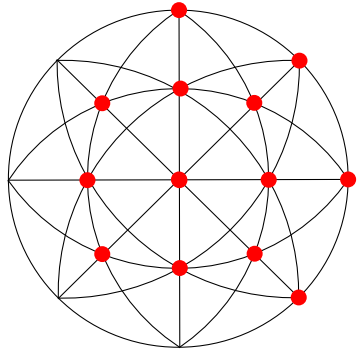
\begin{figure}[H]
				\begin{center}
				\begin{tikzpicture}[line cap=round,line join=round,>=triangle 45,x=1.0cm,y=1.0cm, scale=0.5]
				\clip(0.88,-4.7) rectangle (10.9,4.94);
				\draw (6,0) circle (4.47cm);
				\draw (6.02,4.47)-- (6.04,-4.47);
				\draw (1.53,-0.04)-- (10.47,-0.02);
				\draw [shift={(8.96,0.01)}]  plot[domain=2.15:4.13,variable=\t]({1*5.35*cos(\t r)+0*5.35*sin(\t r)},{0*5.35*cos(\t r)+1*5.35*sin(\t r)});
				\draw [shift={(2.96,-0.01)}]  plot[domain=-0.97:0.97,variable=\t]({1*5.42*cos(\t r)+0*5.42*sin(\t r)},{0*5.42*cos(\t r)+1*5.42*sin(\t r)});
				\draw [shift={(6.01,-2.95)}]  plot[domain=0.58:2.57,variable=\t]({1*5.34*cos(\t r)+0*5.34*sin(\t r)},{0*5.34*cos(\t r)+1*5.34*sin(\t r)});
				\draw [shift={(5.99,3.06)}]  plot[domain=3.75:5.68,variable=\t]({1*5.44*cos(\t r)+0*5.44*sin(\t r)},{0*5.44*cos(\t r)+1*5.44*sin(\t r)});
				\draw  (9.19,3.13)-- (2.85,-3.17);
				\draw  (2.82,3.14)-- (9.17,-3.15);
				\draw [shift={(8.99,-3.01)}] plot[domain=1.54:3.17,variable=\t]({1*6.15*cos(\t r)+0*6.15*sin(\t r)},{0*6.15*cos(\t r)+1*6.15*sin(\t r)});
				\draw [shift={(2.99,3.03)}] plot[domain=-1.59:0.02,variable=\t]({1*6.2*cos(\t r)+0*6.2*sin(\t r)},{0*6.2*cos(\t r)+1*6.2*sin(\t r)});
				\draw [shift={(9.05,3.08)}] plot[domain=3.13:4.73,variable=\t]({1*6.24*cos(\t r)+0*6.24*sin(\t r)},{0*6.24*cos(\t r)+1*6.24*sin(\t r)});
				\draw [shift={(2.97,-3.06)}] plot[domain=-0.02:1.6,variable=\t]({1*6.2*cos(\t r)+0*6.2*sin(\t r)},{0*6.2*cos(\t r)+1*6.2*sin(\t r)});
				\begin{scriptsize}
				\fill [color=red] (3.61,-0.04) circle (6.0pt);
				\fill [color=red] (6.04,-2.37) circle (6.0pt);
				\fill [color=red] (8.38,-0.03) circle (6.0pt);
				\fill [color=red] (4,2) circle (6.0pt);
				\fill [color=red] (6.03,-0.04) circle (6.0pt);
				\fill [color=red] (6.06,2.39) circle (6.0pt);
				\fill [color=red] (8,2.01) circle (6.0pt);
				\fill [color=red] (6.02,-0.02) circle (6.0pt);
				\fill [color=red] (8.01,-1.99) circle (6.0pt);
				\fill [color=red] (4,-2) circle (6.0pt);
				\fill [color=red] (6.02,4.47) circle (6.0pt);
				\fill [color=red] (9.19,3.13) circle (6.0pt);
				\fill [color=red] (10.47,-0.02) circle (6.0pt);
				\fill [color=red] (9.17,-3.15) circle (6.0pt);
				\end{scriptsize}
				\end{tikzpicture}
				\caption{Line arrangement $\overline{\polyA_{4}^d}$ and the points in $\Sing_{ 3}(\overline{\polyA_{4}^d})$ (red bolded), i.e., the singular points with multiplicity at least $3$. In particular, $\Sing_{\geq 4}(\overline{\polyA_{4}^d})$ consists only in the central point.}
				\label{fig:P4 dual with t}
			\end{center}
			\end{figure}
		\end{minipage}
		\end{center}
	\end{example}
	
		\begin{table}[H]
		\centering
		\begin{tabular}{|c|c|c|c|c|c|c|c|c|c|c|}
			\cline{1-3}\cline{5-7}\cline{9-11}
			$\caA_Z$ & $a_Z$ & $b_Z$ & & $\caA_Z$ & $a_Z$ &  $b_Z$ & & $\caA_Z$ & $a_Z$ &  $b_Z$ \\
			
			\cline{1-3}\cline{5-7}\cline{9-11}
		    $A(13,2)$  & 5	&  7   &  &  $ A(20,5)$&  8	 &  11 & &  $ A(29,3) $   & 13	&  15  \\
		    
			\cline{1-3}\cline{5-7}\cline{9-11}
			$A(13,3)$ & 5	 &  7   &    &   $A(21,3)$&  9	&  11   &  & $ A(29,4) $ &13	&  15    \\
			
			\cline{1-3}\cline{5-7}\cline{9-11}
			 $A(17,2)$ &  7	&  9   &    &  $A(21,4)$&  9	&  11   & & $A(29,5) $   & 13	&  15    \\
			 
			\cline{1-3}\cline{5-7}\cline{9-11}
			$A(17,4)$&  7	&  9   &   &   $ A(21,5) $&   9	&  11   & & 	  $ A(30,3) $ & 13	&  16  \\
			
			\cline{1-3}\cline{5-7}\cline{9-11}
			$A(17,3)$ & 7	&  9   &   &  $A(25,2)$ &	11	&  13  & &   $ A(31,2)$  &  13	&  17   \\
			
			\cline{1-3}\cline{5-7}\cline{9-11}
			 $A(19,1)$&  7	&  11   &  &   $ A(25,4) $	   &11	&  13    & &   $ A(31,3) $ &  13	&  17       \\
			 
			\cline{1-3}\cline{5-7}\cline{9-11}
			  $ A(19,3)$&  7	&  11   &  &  $A(25,7)$ & 11	&  13  & &   $A(37,3)$ & 17	&  19   \\
			  			\cline{1-3}\cline{5-7}\cline{9-11}

		\end{tabular}
		\caption{Line arrangements duals to simplicial arrangements defined in \cite{Gru09} and their splitting types. The computations have been made with the algebra software {\it Singular}. The code used can be found as additional file to the arXiv version of the paper and includes the full computation for the $A(31,3)$ case. The coordinates of the points of all the configurations listed in the table have been kindly provided by M. Cuntz during a private communication.}
		\label{tab:A(n,k) sporadic dual}
	\end{table}
	
	In Table \ref{tab:A(n,k) sporadic extra}, we give a list of line arrangements such that, for some $k$, the configuration of points $\Sing_{\geq k}$ admits unexpected curves. We also give the exponents which speak about the degrees of unexpected curves.
	
		\begin{table}[H]
		\centering
		\begin{tabular}{|c|c|c|c|c|}
			\hline
			$\caA$ & $k$ & $|\Sing_{\geq k}(\caA)|$ & $a_{\Sing_{\geq k}(\caA)}$ & $b_{\Sing_{\geq k}(\caA)}$\\
			\hline
			 $A^d(13,2)$   & $4$ & $9$ & 3	&  5  \\
			\hline
			 $A^d(13,2)$   & $3$ & $13$ & 5 &  7  \\
			 \hline
			 $A^d(13,2)$   & $1$ & $25$ & 11 &  13  \\
			 \hline
			 $A^d(17,2)$   & $4$ & $9$ & 3	&  5  \\
			 \hline
			 $A^d(17,4)$   & $4$ & $9$ & 3	&  5  \\
			 \hline
			 $A^d(17,4)$   & $3$ & $25$ & 11	&  13  \\
			 \hline
			 $A^d(19,3)$   & $4$ & $13$ & 5	&  7  \\
			 \hline
			 $A^d(19,3)$   & $3$ & $25$ & 11	&  13  \\
			 \hline
			 $A^d(21,3)$   & $4$ & $13$ & 5	 &  7  \\
			 \hline
			 $A^d(21,3)$   & $3$ & $37$ & 17	&  19  \\
			 \hline
			 $A^d(25,2)$   & $4$ & $21$ & 9	&  11  \\
			 \hline
			 $A^d(25,4)$   & $4$ & $19$ & 7	&  11  \\
			 \hline
		\end{tabular}~~
		\begin{tabular}{|c|c|c|c|c|}
			\hline
		$\caA$ & $k$ & $|\Sing_{\geq k}(\caA)|$ & $a_{\Sing_{\geq k}(\caA)}$ & $b_{\Sing_{\geq k}(\caA)}$\\
			 \hline
			 $A^d(25,7)$   & $4$ & $18$ & 7	&  10  \\
			 \hline
			 $A^d(26,3)$   &  $4$ & $19$  & 7	&  11  \\
			 \hline
			 $A^d(26,4)$   &  $4$ & $18$  & 7	&  10  \\
			 \hline
			 $A^d(27,2)$   &  $4$ & $20$ & 8	&  11  \\
			 \hline
			 $A^d(27,3)$   &  $4$ & $21$  & 8	&  11 \\
			 \hline
			 $A^d(27,4)$   &  $4$ & $19$  & 7	&  11 \\
			 \hline
			 $A^d(28,4)$   &  $4$ & $21$  & 9	&  11 \\
			 \hline
			 $A^d(31,3)$   &  $4$ & $31$  & 13	&  17 \\
			 \hline
			 $A^d(34,2)$   &  $4$ & $13$  & 5	&  7 \\
			 \hline
			 $A^d(37,3)$   &  $4$ & $37$  & 17	&  19 \\
			 \hline
			 $A^d(37,3)$   &  $6$ & $13$  & 5	&  7 \\
			 \hline
		\end{tabular}
		\caption{Configurations of points defined as high order points of some simplicial line arrangements.}
		\label{tab:A(n,k) sporadic extra}
	\end{table}

\subsection{Future directions.}

	The question we considered so far (Problem \ref{question:A}) is a special case of the following more general problem suggested by Cook II, Harbourne, Migliore and Nagel in \cite{CHMN17}.
	
	\begin{problem}\label{question:B}
		Let $Z$ be a set of reduced points in $\PP^2$ and let $\XX = m_1Q_1 + \ldots + m_sQ_s$ be a scheme of fat points with general support. 
		$$
			\text{For which $(Z;m_1,\ldots,m_s;j)$ we have that} \dim_{\CC}[I(Z+\XX)]_j > \max\{\dim_{\CC}[I(Z)]_j - \deg(\XX), 0\}?
		$$
		If so, we say that $Z$ {\bf admits unexpected curves of degree $j$ with respect to $\XX$.}
	\end{problem}
	
	The examples from Table \ref{tab:A(n,k) sporadic dual} and Table \ref{tab:A(n,k) sporadic extra} give examples of unexpected curves according to Problem \ref{question:A} (i.e., when $\XX$ is just a fat point), but we can also extend them to get examples of unexpected curves according to Problem \ref{question:B}. The idea is explained in the following fact.
	
	\begin{proposition}\label{prop: further}
		Let $Z$ be a configuration of points with splitting type $(a_Z,b_Z)$ with $b_Z - a_Z \geq  2$. Then, for any $j \in \{0,\ldots,b_Z-a_Z-2\}$, we have that $Z$ admits a unique unexpected curve of degree $a_Z+1+j$ with respect to $\XX = (a_Z+j)P + A$, where $A$ is a set of reduced points of cardinality $|A| = j$. 
	\end{proposition}
	\begin{proof}
		Let $P$ be a general point. By Proposition \ref{prop:unexpected_curve_unique}, we know that, for any $j \in \{0,\ldots,b_Z-a_Z-2\}$, we have that $Z$ admits unexpected curves of degree $a_Z+j+1$ with respect to $(a_Z+j)P$ and, in particular, that
		$$
			\dim_\CC[I(Z+(a_Z+j)P)]_{a_Z+j+1} = j+1.
		$$
		Moreover, we may observe that, by definition of unexpected curves, we have
		\begin{equation}\label{eq: expected}
			\max\left\{\dim_\CC[I(Z)]_{a_Z+j+1} - {a_Z+j+1 \choose 2}, 0\right\} \leq j.
		\end{equation}
		Now, since generic simple points always impose the expected number of conditions on a linear system of curves, we have that, for any $j \in \{0,\ldots,b_Z-a_z-2\}$,
		$$
			\dim_\CC[I(Z+(a_Z+j)P+A)]_{a_Z+j+1} = \dim_\CC[I(Z+(a_Z+j)P]_{a_Z+j+1}-j = 1;
		$$
		and, at the same time, by \eqref{eq: expected},
		$$
			\max\left\{\dim_\CC[I(Z)]_{a_Z+j+1} - {a_Z+j+1 \choose 2} - j, 0\right\} = 0.
		$$
		Therefore, we have that $Z$ admits a unique unexpected curve of degree $a_Z+j+1$ with respect to $\XX = (a_Z+j)P + A$, where $A$ is a set of generic simple points with $|A| = j$. 
		\end{proof}
	
	Some of the examples provided in Table \ref{tab:A(n,k) sporadic dual} and Table \ref{tab:A(n,k) sporadic extra} satisfy the hypothesis of the latter proposition and give examples of unexpected curves with respect to Problem \ref{question:B}. 
	
	Note that, by Proposition \ref{prop:unexpected_curve_unique}, the unexpected curves constructed in the latter proposition, whenever $j\geq 1$, are reducible. It would be interesting to construct an example of reducible unexpected curve with respect to a scheme of fat points $\XX$ having support in more than one point. 

	Finding a characterization to answer Problem \ref{question:B} for some particular non-connected scheme $\XX$, e.g., the
	union of two fat points, or constructing additional interesting examples of reducible unexpected curve besides the ones constructed in Proposition \ref{prop: further}, are problems worthy of further investigation.
	\bibliographystyle{alpha}
	
\end{document}